\documentclass[11pt,a4paper]{article}%
\usepackage{amsmath, amsthm}
\usepackage{amssymb,latexsym}
\usepackage{amsfonts}

\newtheorem{theorem}{Theorem}%[section]
\newtheorem{corollary}{Corollary}%[section]
\newtheorem{definition}{Definition}%[section]
\newtheorem{example}{Example}%[section]
\newtheorem{lemma}{Lemma}%[section]
\newtheorem{proposition}{Proposition}%[section]
\newtheorem{remark}{Remark}%[section]

\newcommand\tr{\operatorname{tr}}
\def\vol{\operatorname{vol}}
\newcommand\Div{\operatorname{div}}
\def\RR{\mathbb{R}}
\def\NN{\mathbb{N}}
\def\<{\langle}
\def\>{\rangle}
\def\Ric{{\mathcal Ric}}
\def\Ric{\Re}
\newcommand{\eq} {\hspace*{-1.4mm}&=&\hspace*{-1.4mm}}
\newcommand{\plus}{\hspace*{-1.4mm}&+&\hspace*{-1.4mm}}

\newcommand\id{\operatorname{id}}

\begin{document}

\title{On the Bochner technique for singular distributions}
% with statistical structure

\author{
Paul Popescu\thanks{Department of Applied Mathematics, University of Craiova,
Craiova 200585, Str. Al. Cuza, No, 13, Romania.
E-mail address: \texttt{paul$\_$p$\_$popescu@yahoo.com}}\ ,
\
Vladimir Rovenski\thanks{Department of Mathematics, University of Haifa, Mount Carmel, 31905 Haifa, Israel.
\newline
E-mail address: \texttt{vrovenski@univ.haifa.ac.il}}
\ and \
\ Sergey Stepanov\footnote{Department of Mathematics, Finance University, 49-55, Leningradsky Prospect, 125468 Moscow, Russia.
\newline E-mail address: s.e.stepanov@mail.ru}
}

\date{}
\maketitle

\begin{abstract}
In this paper we continue our recent study of a manifold endowed with a singular or regular distribution,
determined as the image of the tangent bundle under a smooth endomorphism,
and generalize Bochner's technique to the case of a distribution with a statistical type structure.
Following the theory of statistical structures on Riemannian manifolds and construction of an almost Lie algebroid on a vector bundle,
we define the modified statistical connection and exterior derivative on tensors. Then we introduce the Weitzenb\"{o}ck type curvature operator on tensors
and derive the Bochner--Weitzenb\"{o}ck type formula. These allow us to obtain vanishing theorems about the null space of the Hodge type Laplacian
on a distribution.



\end{abstract}

\section*{Introduction}

Distributions (subbundles of the tangent bundle) on a manifold are used to build up notions of integrability, and specifically, of a foliation, e.g., \cite{BF,CC,g1967}.
% of a manifold.
There is definite interest of
pure and applied
mathe\-maticians
%, e.g.,~\cite{abt,BL},
to singular distributions and foliations, i.e., having varying dimension, e.g., \cite{BL,Mol}.
Another popu\-lar mathematical concept is a statistical structure, i.e., a Riemannian manifold endowed with a torsionless linear connection
$\widetilde{\nabla}$ such that the tensor $\widetilde{\nabla} g$ is~symmetric in all its entries, e.g., \cite{Amari2016, Mikes1, opozda1, pss-2020, tak, vil}.
The theory of affine hypersurfaces in $\RR^{n+1}$ is a natural source of such manifolds; they also find applications in theory of probability and statistics.

%\begin{definition}\rm
A \textit{singular distribution} ${\cal D}$ on a manifold $M$ assigns to each point $x\in M$ a linear subspace ${\cal D}_x$
of the tangent space $T_xM$ in such a way that, for any $v\in{\cal D}_x$,
there exists a smooth vector field $V$ defined in a neighborhood $U$ of $x$
and such that $V(x) = v$ and $V(y)\in{\cal D}_y$ for all $y$ of~$U$.
A~priori, the dimension of ${\cal D}_x$ depends on $x\in M$. If $\dim{\cal D}_x={\rm const}$,
then ${\cal D}$ is \textit{regular}.
%we obtain a regular distribution.
\textit{Singular foliations} are defined as families of maximal integral submanifolds (leaves) of integrable singular distributions
(certainly, regular foliations correspond to integrable regular distributions).
%\end{definition}

The study of singular distributions is important also because there are plenty of manifolds that do not admit smooth (codimension-one) distributions, while all of them admit such distributions defined outside some ``set of singularities".

Let $M$ be a connected smooth $n$-dimensional manifold,
%$g=\<\cdot,\cdot\>$ -- a Riemannian metric on $M$,
$TM$ -- the tangent bundle,
$\mathcal{X}_{M}$ -- the Lie algebra of smooth vector fields on $M$,
and $\mathrm{End}(TM)$ -- the space of all smooth endomorphisms of $TM$.
% onto itself (linear maps on the fibers of $TM$).
Let $g=\<\cdot,\cdot\>$ be a Riemannian metric on $M$ and $\nabla$ -- the Levi-Civita connection of $g$.

In this paper, we apply the almost Lie algebroid structure (see a short survey in Section~\ref{sec:algebroid}) to singular distributions on $M$,
and in the rest of paper assume $E=TM$ and $\rho=P\in\mathrm{End}(TM)$.

\begin{definition}[see \cite{rp-1}]
%\label{D-singD}
\rm
An image ${\cal D}=P(TM)$ of $TM$ under a smooth endomorphism $P\in\mathrm{End}(TM)$ will
be called a \textit{generalized vector subbundle} of $TM$ or a \textit{singular distribution}.
\end{definition}

\begin{example}\label{Ex-J-phi}\rm
a) Let $P\in\mathrm{End}(TM)$ on $(M,g)$ be of constant rank, $0<r(P)<\dim M$, satisfying
%conditions
\begin{equation*}
%\label{E-ex1b}
 P^2=P,\qquad P^\ast=P,
\end{equation*}
where $P^{\ast}$ is adjoint endomorphism to $P$, i.e., $\<P^{\ast}X,Y\>=\<X,PY\>$,
then we have an \textit{almost product structure} on $(M,g)$, see \cite{g1967}.
In this case, $P$ and $H=\id-P$ are orthoprojectors onto vertical distribution
$P(TM)$ and  horizontal distribution $H(TM)$,
which are complementary orthogonal and regular, but none of which is in general integrable.
Many popular geometrical structures belong to the case of almost product structure,
e.g.,
%submersions,
$f$-structure (i.e., $f^3+f=0$)
%almost contact structure
 and para-$f$-structure (i.e., $f^3-f=0$); such structures on singular distributions were considered in~\cite{rp-2}.
Almost product structures on statistical manifolds $(M,g,\widetilde\nabla)$ were studied in~\cite{tak, vil}.

b) The case $P=J$, where $J^2=-\id$, has nothing to do with distributions: $P(TM)=TM$,
but it gives us an \textit{almost complex structure} on $(M,g)$ and ${\cal D}=TM$.
For \textit{integrable} structure defined by $P$ we have $N_{P}=0$,
and the particular case $P=J$ gives an integrable almost complex structure.
The \textit{Nijenhuis tensor} $N_{P}$ of
%the endomorphism
$P$ is defined~by
\begin{eqnarray*}
 && N_{P}(X,Y) = [PX,PY] -P[PX,Y] -P[X,PY] +P^{2}[X,Y] \\
 && =\,[PX,PY]-P(\nabla_{PX}Y +(\nabla_{X}P) Y) +P(\nabla_{PY}X -(\nabla_{Y}P) X).
\end{eqnarray*}

c) Let ${\cal F}$ be a singular Riemannian foliation of $(M,g)$, i.e., the leaves are smooth, connected, locally equidistant submanifolds of $M$. e.g., \cite{Mol}.
Then $T{\cal F}$ is a singular distribution parameterized by the orthoprojector $P:TM\to T{\cal F}$.
\end{example}

In this article, we generalize Bochner's technique to a Riemannian manifold endowed with a singular (or regular) distribution and
a statistical type connection,
%the differential
%geometry of a Riemannian manifold endowed with a singular
%(or regular)
%distribution and a statistical connection,
continue our study~\cite{PP0,PMAlg,PP01,P2,rp-1} and generalize some results of other authors in~\cite{opozda1,Peter2,rp-2}.
% Our~objective is to .
%a Riemannian manifolds endowed with singular distributions and statistical structures.
Recall that the Bochner technique works for skew-symmetric tensors lying in the kernel of the Hodge Laplacian
$\Delta_H=d\,\delta+\delta\,d$
on a closed manifold: using maximum principles, one proves that such tensors are parallel, e.g.,~\cite{Peter,Peter2}.
%The Bochner technique works for skew-symmetric tensors lying in the kernel of
%the {Hodge Laplacian} $\Delta_H=d\,\delta+\delta\,d$ on a closed manifold,
%using maximum principles, they prove that such tensors are parallel.
Here $d$ is the exterior differential operator, and $\delta$ is its adjoint operator for
%with respect to
the $L^2$ inner product,
is an elliptic differential operator. It~can be decomposed into two terms,
%e.g., \cite[Theorem~9.4.1]{Peter},
\begin{equation}\label{E-Wei0}
 \Delta_H =\nabla^\ast\nabla+\Ric,
\end{equation}
one is the \textit{Bochner Laplacian} $\nabla^\ast\nabla$, and the second term
(depends linearly on the Riemannian curvature tensor)
is called the \textit{Weitzenb\"{o}ck curvature operator} on $(0,k)$-tensors $S$,
%over $(M,g)$,
e.g.,~\cite{Peter}. %%[Section~9.3.2]
\begin{eqnarray}\label{E-Ric}
%\nonumber
 && \Ric\,(S)(X_{1},\ldots,X_{k}) = \sum\nolimits_{a=1}^{k}\sum\nolimits_{i=1}^{n}
 (R_{\,e_{i},X_{a}}\,S)(\underbrace{X_{1},\ldots, e_{i}}_{a},\ldots, X_{k}).
\end{eqnarray}
Here $\nabla^\ast$ is the $L^2$-adjoint
%operator
of the Levi-Civita connection $\nabla$,
and $R$ acts on $(0,k)$-tensors by
\begin{equation}\label{Curv-S-1}
  (R_{X,Y}\,S)(X_1,\ldots,X_k) = -\sum\nolimits_{\,i}S(X_1,\ldots R_{X,Y}X_i,\ldots,X_k).
\end{equation}
The Weitzenb\"{o}ck decomposition formula \eqref{E-Wei0} allows us to extend the Hodge Laplacian to arbitrary tensors
and is important in the study of interactions between the geometry and topology of manifolds.

The work has the Introduction and eight sections, the References include 16 items.
In~Sections~\ref{sec:nabla}, \ref{sec:div} and \ref{sec:laplace},
following an almost Lie algebroid construction (Section~\ref{sec:algebroid} with Appendix)
and concept of statistical structure (Section~\ref{sec:stat}),
we~define the derivatives $\nabla^P$ and $d^P$, the modified divergence and their $L^2$ adjoint operators on tensors,
and modified Laplacians on tensors and forms.
In~Section~\ref{sec:R}, using $\nabla^P$ and making some assumptions about $P$
(which are trivial when $P=\id_{TM}$), we define the curvature operator $R^P$.
In~Section~\ref{sec:Ric}, we define the Weitzenb\"{o}ck type curvature operator on tensors, prove
the Bochner--Weitzenb\"{o}ck type formula and obtain vanishing results.
The assumptions that we use are reasonable, as illustrated by examples.

\section{The modified covariant derivative and bracket}
\label{sec:nabla}

Here, we~define the map  $\nabla^P:\mathcal{X}_{M}\times\mathcal{X}_{M}\to\mathcal{X}_{M}$, called $P$-\textit{connection},
which depends on $P$ and a $(1,2)$-tensor $K$ (called \textit{contorsion tensor}), and generally is not a linear connection on $M$,
\begin{equation}\label{nablaP}
 \nabla_{X}^{P}\,Y=\nabla_{PX}\,Y+K_XY.
\end{equation}
Set $\nabla_{X}^{P}f=(PX)f$ for $f\in C^1(M)$ (the $P$-\textit{gradient} of $f$) and notice that $\nabla^{P}$ satisfies axioms \eqref{E-rho-conn} in Section~\ref{sec:algebroid}.
In~particular, for $K=0$, we have the $P$-{connection} $\widehat\nabla^{P}$ defined in \cite{rp-2} by
\begin{equation}\label{nablaP-hat}
 \widehat\nabla_{X}^{P}\,Y=\nabla_{PX}Y,
\end{equation}
and playing an important role in our study.
Using $\nabla^{P}$, we construct the ${P}$-derivative of $(s,k)$-tensor $S$, where $s=0,1$,
 as $(s,k+1)$-tensor $\nabla^{P} S$:
\begin{equation}\label{E-nablaP-S}
 (\nabla^{P} S)(Y,X_{1},\ldots,X_{k}) = \nabla_{Y}^{P}(S(X_{1},\ldots,X_{k}))
 -\sum\nolimits_{i=1}^{k}S(X_{1},\ldots,\nabla_{Y}^{P}X_{i},\ldots,X_{k}) .
\end{equation}
%This also makes sense for $s\ge2$, using the product rule for tensors:
%\[
% \nabla^{P}_X(S_1\otimes S_2) = \nabla^{P}_X S_1\otimes S_2+S_1\otimes\nabla^{P}_X S_2 .
%\]
We use the standard notation $\nabla^{P}_Y\,S=\nabla^{P} S(Y,\ldots)$.
A tensor $S$ is called $P$-\textit{parallel} if $\nabla^{P}S = 0$.
%If $S$ is a $(0,k)$-tensor
%$\omega\in\Lambda^{k}(M)$ is a $k$-form,
% then $\nabla^{P}S$ is a $(0,k+1)$-tensor
%\begin{equation}\label{E-nablaP-omega}
% (\nabla^{P} S)(Y,X_{1},\ldots,X_{k})
% ={P}Y(S(X_{1},\ldots,X_{k})) -\sum\nolimits_{i=1}^{k} S(X_{1},\ldots,\nabla_{Y}^{P}X_{i},\ldots,X_{k}).
%\end{equation}
%Note that $\nabla^{P}_Y\,S\ne \nabla_{PY}\,S$ when $K\ne0$.

A linear connection $\widetilde{\nabla} = \nabla + {K}$ on a Riemannian manifold $(M,g)$
is \textit{metric compatible} if $\widetilde{\nabla} g=0$; in~this case, $K_X^* = -K_X$,
where $K^*_X$ is adjoint to $K_X$ with respect to $g$.
%For {Riemann--Cartan spaces}, the contorsion tensor $\mathfrak{T}_X\ (X\in TM)$ is anti-symmetric:
%\begin{equation*}
%\label{E-nabla-g}
% (\widetilde{\nabla}_X\,g)(Y,Z) = \<K_XY,Z\> + \<K_XZ,Y\> =0 \quad (X,Y,Z\in TM).
%\end{equation*}
%see Section~\ref{ssec:riemann}.
This concept can be applied for $P$-connections.
Recall that $\widehat\nabla^{\,P}$ is metric compatible, see \cite{rp-2}.

\begin{proposition}
The $P$-{connection} has a metric property, i.e., $\nabla^{P} g=0$, if and only if the map $K_X\in\mathrm{End}(TM)$ is skew-symmetric for any $X\in TM$, that is
%\[
 $\<K_XY, Z\> = - \<K_XZ, Y\>$.
%\]
\end{proposition}

\begin{proof} We calculate using \eqref{E-nablaP-S},
\begin{equation}\label{E-proof-K}
 (\nabla^{P}_X\,g)(Y,Z) = (\nabla_{PX}\,g)(Y,Z) - \<K_XY, Z\> - \<K_XZ, Y\>.
\end{equation}
Since $\nabla$ has the metric property, then $\nabla_{PX}\ g=0$, and the claim follows.
\end{proof}

%%%%%%%%%%%%%%%%%%%%%%%%

Using \eqref{nablaP}, define a skew-symmetric ${P}$-\textit{bracket} $[\cdot,\cdot]_{P}:\mathcal{X}_{M}\times\mathcal{X}_{M}\rightarrow\mathcal{X}_{M}$ by
\begin{equation}\label{E-Pbracket}
 [X,Y]_{P}=\nabla_{X}^{P}\,Y-\nabla_{Y}^{P}\,X .
\end{equation}
By \eqref{E-Pbracket} and according to definition \eqref{E-T-rho}$_1$ in Section~\ref{sec:algebroid}, the $P$-connection $\nabla^{P}$ is torsion free.
%Given bracket \eqref{E-Pbracket}, we define the following operators, see
According to \eqref{E-J-rho} and \eqref{E-D-rho}, we use the bracket \eqref{E-Pbracket} to define the following operator:
%\begin{definition}\label{D-rho}\rm
%Define operator
%s (the first one is called the \emph{Jacobiator} of $[\cdot,\cdot]_{P}$)
%\begin{align*}
% & J_{P}(X,Y,Z) = \sum\nolimits_{\mathrm{cycl.}}[\,[X,Y]_{P},Z]_{P},\\
\[
 {\mathfrak{D}}^{P}(X,Y) = [{P}X, {P}Y] - {P}[X,Y]_{P}.
\]
%\end{align*}
%\end{definition}

%\begin{lemma}
%We have for $\nabla^{P}$ in \eqref{nablaP},
%\[
% {\mathfrak{D}}^{P}(X,Y)=(\nabla_{PX}P)(Y)-(\nabla_{PY}P)(X)-P(K_XY) +P(K_YX).
%\]
%\end{lemma}

%\begin{proof}
%Using \eqref{E-Pbracket}, we have
%\[
% [X,Y]_{P}=\nabla_{X}^{P}Y-\nabla_{Y}^{P}X=\nabla_{PX}Y-\nabla_{PY}X+K_XY-K_YX.
%\]
%Thus,
%\begin{eqnarray*}
% {\mathfrak{D}}^{P}(X,Y) \eq \nabla_{PX}PY -P\nabla_{PX}Y -P(K_XY) -(\nabla_{PY}PX -P\nabla_{PY}X -P(K_YX))\\
% \eq (\nabla_{PX}P)Y-(\nabla_{PY} P)X -P(K_XY) +P(K_YX),
%\end{eqnarray*}
%and the conclusion follows.
%\end{proof}

Note that the equality $\mathfrak{D}^{P}=0$ corresponds to \eqref{E-anchor}$_3$
%the third axiom of Definition~\ref{D-ALA}
with $\rho=P$
%(see Section~\ref{sec:algebroid})
of a skew-symmetric bracket.
The following result generalizes \cite[Proposition~3]{PP01}.

\begin{proposition}\label{prAlg01}
Condition $\mathfrak{D}^{P}=0$
%, see Definition~\ref{D-rho},
is equivalent to the symmetry on covariant components of the $(1,2)$-tensor
${\cal A}(X,Y)=(\nabla_{PX}P)(Y)-P(K_XY)$, that is
\begin{equation}\label{E-condPP}
 (\nabla_{PX}P)(Y) - P(K_XY) = (\nabla_{PY}P)(X) -P(K_YX).
\end{equation}
%In this case, the $P$-bracket \eqref{E-Pbracket} defines a skew-symmetric algebroid structure  with $E=TM$ and $\rho=P$,
%see \cite{rp-2}.
\end{proposition}

\begin{proof}
Using \eqref{E-Pbracket}, we have
\begin{equation}\label{E-Pbracket-b}
 [X,Y]_{P}
 %=\nabla_{X}^{P}Y-\nabla_{Y}^{P}X
 =\nabla_{PX}Y-\nabla_{PY}X+K_XY-K_YX.
\end{equation}
Thus,
\begin{eqnarray*}
 {\mathfrak{D}}^{P}(X,Y) \eq \nabla_{PX}PY -P\nabla_{PX}Y -P(K_XY) -\nabla_{PY}PX +P\nabla_{PY}X +P(K_YX)
 %[PX,PY] -P[X,Y]_{P}
  \\
 \eq \nabla_{PX}PY-\nabla_{PY}PX-P(\nabla_{X}^{P}\,Y)+P(\nabla_{Y}^{P}\,X) \\
 \eq {\cal A}(X,Y) -{\cal A}(Y,X),
\end{eqnarray*}
and the conclusion follows.
\end{proof}

\begin{theorem}
If \eqref{E-condPP} holds for a
%torsion free
$P$-connection \eqref{nablaP}, then the anchor $P$ and the bracket $[\cdot,\cdot]_{P}$ given in \eqref{E-Pbracket} define a
skew-symmetric algebroid structure on $TM$.
% and $\rho=P$.
\end{theorem}

\begin{proof} This follows from Proposition~\ref{prAlg01}, according to Definition~\ref{D-ALA} in Section~\ref{sec:algebroid}.
\end{proof}

%Examples with \eqref{E-condPP} when $K=0$ given in \cite{rp-2} can be extended for the case $K\ne0$.

\begin{example}
%\label{Ex2a}
\rm
%Consider the tensor $K=c\,\nabla P$.
% from Example~\ref{Ex01}.
%the tensor $P$\ has its Nijenhuis tensor
If $N_{P}=0$ and $K=c\,\nabla P$, then
%the tensor
${\cal A}$
%given in Proposition~\ref{prAlg01}
is symmetric,
thus the condition \eqref{E-condPP} holds.
%Notice that this agrees with Example~\ref{Ex01}b).
%\end{example}
%\begin{example}\rm
%b)~
%\end{example}
%\begin{example}\rm
%The~case $N_J=0$ gives an integrable almost product structure.
\end{example}

\section{The statistical $P$-structure}
\label{sec:stat}

If a linear connection $\widetilde{\nabla}$ on a Riemannian manifold $(M,g)$ is torsionless and tensor $\widetilde{\nabla} g$ is symmetric in all its entries then
$\widetilde{\nabla}$ is called a \emph{statistical connection}, e.g., \cite{Amari2016,opozda1};
and the pair $(g,\widetilde\nabla)$ is called a \textit{statistical structure} on $M$.
%Its contorsion tensor has the following symmetries:
In~this case,
%$K_YX = {K}_XY$ and ${K}_X^* = {K}_X$,
\begin{equation}\label{E-stat-K}
 K^*_X = K_X,\quad K_XY=K_YX\quad(X,Y\in TM),
\end{equation}
equivalently, the \textit{statistical cubic form} $A(X,Y,Z)=\<K_X Y,Z\>$ is symmetric.
 We introduce a similar concept for
%$P$-connections on
singular distributions.

\begin{definition}
%\label{D-02}
\rm
The $\nabla^{P}$ will be called a \textit{statistical $P$-connection} on $(M,g)$ if
the statistical cubic form $A(X,Y,Z)$ is symmetric, or, equivalently, \eqref{E-stat-K} holds. In this case,
the pair $(g,\nabla^{P})$ is called a \textit{statistical $P$-structure} on $M$.
%The $\nabla^{P}$ will be called \textit{equiaffine} if there is a {volume form} $d\vol$ such that $\nabla^{P} d\vol =0$.
\end{definition}

\begin{proposition}\label{P-AXYZ}
If $\,\nabla^{P}$ is a statistical $P$-connection for $g$ then the (3,0)-tensor $\nabla^{P} g$ is symmetric in all its entries,
i.e., the following Codazzi type condition~holds:
\begin{equation}\label{E-statP}
 (\nabla^{P}_X\,g)(Y,Z) = (\nabla^{P}_Y\,g)(X,Z) = (\nabla^{P}_X\,g)(Z, Y).
\end{equation}
%(3,0)-tensor $A(X,Y,Z):=\<K_X Y,Z\>$ (called
%the statistical cubic form is symmetric, or, equivalently,
%\begin{equation}\label{E-stat-K}
% K^*_X = K_X,\quad K_XY=K_YX\quad(X,Y\in TM),
%\end{equation}
%where $K^*_X$ is adjoint to $K_X$ with respect to $g$.
\end{proposition}

\begin{proof}
By \eqref{E-proof-K}, \eqref{E-stat-K} and the property $\nabla g=0$, we have $(\nabla^{P}_X\,g)(Y,Z) = -2A(X,Y,Z)$,
thus all three terms in \eqref{E-statP} are equal.
%This follows from \eqref{E-Pbracket}, \eqref{E-condPP}, the property $\nabla g=0$ and Definition~\ref{D-02}.
\end{proof}

Since $\nabla_{PX}\,g =0$ for the Levi-Civita connection, condition \eqref{E-statP}
does not impose restrictions on $P$ and it is equivalent to the property ``the cubic form $A$ is totally symmetric".

By \eqref{E-Pbracket-b} and \eqref{E-stat-K}, the $P$-bracket of a statistical $P$-structure does not depend on $K$:
\begin{equation}\label{E-bracket-stat}
 [X,Y]_{P} =\nabla_{PX}Y-\nabla_{PY}X.
\end{equation}
If $\nabla^P$ is statistical then $\widehat\nabla_{X}^{P}$, see \eqref{nablaP-hat}, has the same $P$-bracket and $\widehat{\mathfrak{D}}^{\,P}={\mathfrak{D}}^{P}$.
Proposition~\ref{prAlg01} yields the following result for a statistical $P$-structure.

\begin{corollary}
%\label{C-Alg01}
For a statistical $P$-structure, condition $\mathfrak{D}^{P}=0$, see \eqref{E-condPP}, is equivalent to
%the symmetry on covariant components of the $(1,2)$-tensor
%${\cal A}(X,Y)=(\nabla P)(PX,Y)-P(K_XY)$, that is
\begin{equation}\label{E-condPP-stat}
 (\nabla_{PX}P)(Y) = (\nabla_{PY}P)(X),\quad X,Y\in{\mathcal X}_{M} ,
\end{equation}
\end{corollary}

\begin{proof}
We can put ${\cal A}(X,Y)=(\nabla P)(PX,Y)$ and reduce \eqref{E-condPP} to a simpler view \eqref{E-condPP-stat}.
\end{proof}

%%% NEW
The notion of conjugate connection is important for statistical manifolds, see \cite{opozda1,sss-1}.

\begin{definition}\rm
For a $P$-connection $\nabla^P$ on $(M,g)$, its \textit{conjugate $P$-connection} $\bar\nabla^P$ is defined by the following equality:
\[
 PX\<Y,Z\> = \<\nabla^{P}_X Y, Z\> + \<Y, \bar\nabla^{P}_X Z\>.
\]
\end{definition}

One may show that $\bar\nabla^{P}_X = \widehat\nabla^{P}_{X} -K^*_X$ in general,
thus, for a statistical $P$-connection $\nabla^{P}$ the conjugate connection $ \bar\nabla^{P}$ is given by
\[
 \bar\nabla^{P}_X = \bar\nabla^{P}_{X} - K_X .
\]
%and $\nabla_{PX}Y = \frac12\,(\nabla^{P}_X Z+\bar\nabla^{P}_X Z)$ holds.
In turn,
%according to the above,
the statistical $P$-connection $\nabla^{P}$ is conjugate to $\bar\nabla^{P}$.
Note that $2\,\widehat\nabla^{P}=\nabla^{P}+\bar\nabla^{P}$.

\begin{remark}\label{R-02}\rm
For a conjugate statistical $P$-connection $\bar\nabla^{P}$, we can define the $P$-bracket by
$\overline{[X,Y]}_{P}=\bar\nabla_{X}^{P}\,Y-\bar\nabla_{Y}^{P}\,X$ and the tensor
$\bar{\mathfrak{D}}^{P}(X,Y) = [{P}X, {P}Y] - {P}\overline{[X,Y]}_{P}$.
%and $\bar A$
%by
%\[
%  \bar{\mathfrak{D}}^{P}(X,Y) = [{P}X, {P}Y] - {P}\overline{[X,Y]}_{P} .
%%  ,\quad \bar{\cal A}=(\nabla_{PX}{P})Y+P(K_XY).
%\]
By~\eqref{E-stat-K},
%for the conjugate statistical $P$-structure
we have
\[
 \overline{[\cdot\,,\cdot\,]}_{P}=[\cdot\,,\cdot\,]_{P},\quad
 \bar{\cal A}={\cal A},\quad
 \bar J_P=J_P,\quad
 \bar{\mathfrak{D}}^{P}={\mathfrak{D}}^{P}.
\]
\end{remark}

From Proposition~\ref{P-AXYZ}, using Remark~\ref{R-02}, we obtain the following corollaries.

\begin{corollary}
 The pairs $(g,\nabla^{P})$ and $(g,\bar\nabla^{P})$ are simultaneously statistical $P$-structures on $M$.
\end{corollary}

%\begin{proof}  \end{proof}

\begin{corollary}
A  statistical $P$-structure on $(M,g)$ and its conjugate simultaneously define skew-symmetric algebroid structures on $TM$.
\end{corollary}

%\begin{proof}
%If the tensors $(\nabla g)(PX,Y)$ and $P(K_X Y)$ are symmetric in $X$ and $Y$ then $(\nabla^{P},[\cdot\,,\cdot\,]_P)$ and %$(\bar\nabla^{P},\overline{[\cdot\,,\cdot\,]}_P)$ define skew-symmetric algebroid structures on $TM$.
%\end{proof}

To simplify the calculations, for the rest of this article we will restrict ourselves to statistical $P$-structures, see \eqref{E-stat-K},
and to use the concept of almost Lie algebroid, assume~\eqref{E-condPP-stat}.
% and \eqref{E-cond-PP-stat}.

Define the vector field $E=\sum\nolimits_{\,i} K_{e_{i}}e_{i}$.
%=\operatorname*{trace}\nolimits_{g}(K_{\cdot}\cdot)$,
Using \eqref{E-stat-K}, we get
\[
 \<E,X\>=\tr_g K_X,\quad X\in\mathfrak{X}_M.
\]
For any $(k+1)$-form $\omega$, set
\[
 (K_{Y}\,\omega)(X_1,X_2,\ldots,X_k) = -\sum\nolimits_{\,i}\omega(X_{1},\ldots,K_{Y}X_{i},\ldots,X_{k}).
\]

\begin{lemma}[see Lemmas~6.2 and 6.3 in \cite{opozda1}]
For any local orthonormal frame $\{e_{i}\}$ and any $k$-form $\omega$ we have
\begin{equation}\label{E-stat-L1-a}
 \sum\nolimits_{\,i}(K_{e_{i}}\,\omega)(e_{i},X_2,\ldots,X_k) = -\iota_{\,E}\,\omega(X_1,\ldots,X_k),
\end{equation}
and for any $(k+1)$-form, $k\ge1$, and an index $a\in\{1,\ldots,k\}$ be fixed, we have
\begin{equation}\label{E-stat-L1-b}
 \sum\nolimits_{\,i}\omega(e_{i},X_{1},\ldots,K_{e_{i}}X_{a},\ldots,X_{k}) = 0.
\end{equation}
\end{lemma}

\section{The modified divergence}
\label{sec:div}

Define the $P$-divergence of a vector field $X$ on $(M,g)$ using a local orthonormal frame $\{e_{i}\}$~by
\begin{equation}\label{E-Pdiv-1}
 \Div_P X
 = \operatorname{trace}(Y{\to}\,\nabla^{P}_{Y}\,X)
 = \sum\nolimits_{\,i}\<\nabla^P_{e_i} X,\, e_i\>.
\end{equation}
In order to generalize the Stokes Theorem for distributions, we formulate the following.

\begin{lemma}
%[see Proposition~2.4 in \cite{rp-1}]\label{L-03}
On a Riemannian manifold $(M,g)$ with a statistical $P$-structure,
%Given $P\in\mathrm{End}(TM)$ and $(1,2)$-tensor $K$ on $(M,g)$,
the
%system of
condition
%For a statistical $P$-structure $(g,\nabla^{P})$,
% the curvature tensors for $\nabla^{P}$ and its conjugate connection are the same.
%and \eqref{E-cond-PP} reads as
\begin{equation}\label{E-cond-PP-stat}
 ({\rm div}\,P)(X) = \tr_g K_X,
 %\<E,X\>
 \quad X\in{\mathcal X}_{M}
\end{equation}
is equivalent to the following equality:
\begin{equation}\label{E-divf-4}
 {\rm div}_{P}\,X={\rm div} (PX),\quad X\in{\mathcal X}_{M}.
\end{equation}
\end{lemma}

\begin{proof}  Note that
\begin{eqnarray*}
 && \sum\nolimits_{\,i}\<\nabla_{Pe_i} X, e_i\> =
 \sum\nolimits_{\,i,j}\<Pe_i, e_j\>\<\nabla_{e_j} X, e_i\> =
 \sum\nolimits_{\,i,j}\<e_i, P^*e_j\>\<\nabla_{e_j} X, e_i\> \\
 && =\sum\nolimits_{\,j}\<\nabla_{e_j} X, P^*e_j\>
 = \sum\nolimits_{\,j}\<P\nabla_{e_j} X, e_j\> = \Div(PX) - (\Div P)(X).
\end{eqnarray*}
Using this, definition \eqref{nablaP} and \eqref{E-stat-K}, we have
\begin{equation*}
 \Div_P X =
 %\sum\nolimits_{\,i}\<\nabla^P_{e_i} X,\, e_i\> =
 \sum\nolimits_{\,i}\,\<\nabla_{Pe_i} X + K_{e_i}\,X, e_i\>
 %+\tr_g(Y\rightarrow K_Y X)
 %\tr_g K_X  \\
 = \Div(PX) - (\Div P)(X) +\tr_g K_X.
\end{equation*}
From this and \eqref{E-stat-K} the claim follows.
\end{proof}
%\begin{equation}\label{E-cond-PP}
% ({\rm div}\,P)(X) - \tr(Y\rightarrow K_Y X) =0,\quad X\in{\mathcal X}_{M},
%\end{equation}
%By \eqref{E-stat-K}$_2$, condition \eqref{E-cond-PP} reduces to  \eqref{E-cond-PP-stat}.

%From Proposition~\ref{P-03} we obtain the generalization of \textbf{Divergence Theorem},
%which for $P=\id_{TM}$ reduces to the classical one.

\begin{theorem}\label{T-P-Stokes}
%Thus, Lemma~\ref{L-03} allows us to extend the Divergence Theorem as follows:
%, see \eqref{E-Div-Th}:
Let a statistical $P$-structure on a compact Riemannian manifold $(M,g)$ with boundary satisfies \eqref{E-cond-PP-stat}.
Then for any $X\in\mathfrak{X}_M$  we have
\begin{equation*}
%\label{E-Div-Th-P1}
 \int_M (\Div_P X)\,{\rm d}\vol_g = \int_{\partial M}\<X, P(\nu)\>\,d\omega,
\end{equation*}
where, as in the classical case, $\nu$ is the unit inner normal to $\partial M$.
In particular, on a Riemannian manifold $(M,g)$ without boundary,
for any $X\in\mathfrak{X}_M$ with compact support, we have
\begin{equation*}
%\label{E-Div-Th-P}
 \int_M (\Div_P X)\,{\rm d}\vol_g = 0.
\end{equation*}
\end{theorem}

%By Theorem~\ref{T-P-Stokes},

\begin{example}
%\label{Ex01}
\rm
For the tensor $K_XY = (\Div P)(Y)\cdot X$ where $X,Y\in TM$,
% $K_XY=(\Div  P)(X)\,Y$.
 the property \eqref{E-cond-PP-stat} follows from ${\rm div}\,P=0$.
%\end{example}
%\begin{example}\label{Ex02}\rm
The same holds for a more general (1,2)-tensor $K=c\,\nabla P$ with any $c\in\RR$.
 %$K_XY=(\nabla_{X}P)(Y)$.
% also satisfies \eqref{E-cond-PP-stat}$_2$.
\end{example}

The following pointwise inner products and norms for $(0,k)$-tensors will be used:
\[
 \<S_1,\,S_2\> = \sum\nolimits_{\,i_1,\ldots,\,i_k} S_1(e_{i_1},\ldots,\,e_{i_k})\,S_2(e_{i_1},\ldots,\,e_{i_k}),\quad
  \|S\| = \sqrt{\<S,\,S\>}
\]
while, for $k$-forms, we set
\[
 \<\omega_1,\,\omega_2\> = \sum\nolimits_{\,i_1<\ldots\,<i_k} \omega_1(e_{i_1},\ldots,\,e_{i_k})\,\omega_2(e_{i_1},\ldots,\,e_{i_k}).
\]
For $L^2$-product of compactly supported tensors on a Riemannian manifold, we set
\[
 (S_1,\,S_2)_{L^2}=\int_M\<S_1,\,S_2\>\,{\rm d}\vol_g.
\]
 The following $\nabla^{\ast{P}}$ maps $(s,k+1)$-tensor, where $s=0,1$, to $(s,k)$-tensor:
\begin{equation*}
%\label{E-nabla-astP}
 (\nabla^{\ast{P}} S)(X_{1},\ldots,X_{k})
 = -\sum\nolimits_{i} (\nabla^{P}_{e_{i}}S)(e_{i},X_{1},\ldots,X_{k}) ,
\end{equation*}
and similarly for $\bar\nabla^{\ast{P}}$ and $\widehat\nabla^{\ast{P}}$.
%the $L^2$-{adjoint to the $P$-derivative} on $k$-forms,
%tensors,
%denoted by $\nabla^{\ast{P}}$,
%is given by
%$\nabla^{\ast{P}} S$
%For a $1$-form $\omega$, we have
%\begin{equation*}
% \nabla^{\ast{P}}\omega = \widehat\nabla^{\ast{P}}\omega +\omega(E),\quad
% \bar\nabla^{\ast{P}}\omega = \widehat\nabla^{\ast{P}}\omega -\omega(E).
%\end{equation*}
Using \eqref{E-stat-L1-b}, we relate $\nabla^{\ast{P}}$ and $\widehat\nabla^{\ast{P}}$ for any $k$-form $\omega$:
%from \eqref
\begin{equation}\label{E-adj-nabla-P}
 \nabla^{\ast{P}}\omega = \widehat\nabla^{\ast{P}}\omega +\iota_E\,\omega,\quad
 \bar\nabla^{\ast{P}}\omega = \widehat\nabla^{\ast{P}}\omega -\iota_E\,\omega.
\end{equation}
Thus, $\bar\nabla^{\ast{P}}\omega = \nabla^{\ast{P}}\omega -2\,\iota_E\,\omega$.
%The ``musical" isomorphisms $\sharp:T^*M\to TM$ and $\flat: TM\to T^*M$ are used for rank one tensors, e.g. if $\omega\in T^1_0(M)$ is a 1-form %and $X\in {\mathfrak X}_M$ then $\omega(X)=\<\omega^\sharp,X\>=\<\omega,X^\flat\>=X^\flat(\omega^\sharp)$.
The $\nabla^{\ast{P}}$ is related to the ${P}$-\textit{divergence} \eqref{E-Pdiv-1} of
%a~vector field
$X\in{\cal X}_{M}$~by
\begin{equation}\label{E-divPX}
 {\rm div}_{P}\,X  = -\nabla^{\ast{P}} X^\flat,
\end{equation}
where $X^\flat$ is the 1-form dual to $X$.

To simplify the calculations and use the results of \cite {rp-2} with $\widehat\nabla^{\,P}$,
we will also consider statistical $P$-structures with stronger conditions than \eqref{E-cond-PP-stat},
\begin{equation}\label{E-cond-PP-stat-2}
 {\rm div}\,P = 0,\quad E = 0 .
\end{equation}

\begin{example}[see \cite{rp-2}]\rm One can use structures mentioned in Example~\ref{Ex-J-phi} to clarify
%the property
\eqref{E-cond-PP-stat-2}(a).

(a) For an almost complex structure $P=J$ on $TM$, see Example~\ref{Ex-J-phi}(b),
the property \eqref{E-cond-PP-stat-2}(a) describes a class of \textit{almost Hermitian manifolds}
which includes K\"{a}hlerian manifolds, i.e., $\nabla J=0$.
Differentiating $J^2=-\id_{\,TM}$, we obtain
\begin{equation}\label{E-J-near}
 (\nabla_{X}\,J)J=-J(\nabla_{X}\,J),\quad X\in\mathfrak{X}_M .
\end{equation}
By \eqref{E-J-near}, our class contains a wider
class of \textit{nearly K\"{a}hlerian manifolds},
which are defined by $(\nabla_X J)X=0$.
%, see \cite{g1976}.
%Indeed, if $(M,J,g)$ is nearly K\"{a}hlerian, then
%\begin{eqnarray*}
% 0 \eq (\nabla_{J(X+Y)}J)J(X+Y)
% =(\nabla_{JX}\,J)JX+(\nabla_{JX}\,J)JY+(\nabla_{JY}\,J)JX+(\nabla_{JY}\,J)JY \\
% \eq (\nabla_{JX}\,J)JY+(\nabla_{JY}\,J)JX = -J\big((\nabla_{JX}\,J)Y+(\nabla_{JY}\,J)X\big)\,
%\end{eqnarray*}
%from which \eqref{E-cond-PP-stat-2}(a) follows.
There are many nearly K\"{a}hlerian manifolds that are not K\"{a}hlerian.

(b) An $f$-structure on $M$ generalizes the almost complex and the almost contact structures.
The restriction of $f$ to ${\mathcal D}=f(TM)$ determines a complex structure on it.
An interesting case of $f$-structure on $M^{2n+p}$ occurs when $\ker f$ is parallelizable for which there
exist global vector fields $\xi_i,\ i\in\{1,\ldots, p\}$, with their dual 1-forms $\eta^i$, satisfying
%(see~\cite{gy})
the following relations:
\begin{equation*}
%\label{E-f2}
 f^2 = -\id_{TM} +\sum\nolimits_i\eta^i\otimes\xi_i,\quad \eta^i(\xi_j)=\delta^i_j.
\end{equation*}
It is known that that $f\,\xi_i=0$, $\eta_i\circ f=0$ and $f$ has rank $2n$.
A Riemannian metric $g=\<\cdot,\cdot\>$ is compatible, if $f^*f =\id_{TM} -\sum\nolimits_i\eta^i\otimes\xi_i$.
We have $f^*=-f$, and for $P=f$, we get
\begin{eqnarray*}
 ({\rm div}\,P\,P^*)(X) =-\sum\nolimits_j\<\nabla_{\xi_j}\,\xi_j + ({\rm div}\,\xi_j)\xi_j,\ X\>.
\end{eqnarray*}
Thus, \eqref{E-cond-PP-stat-2}(a) holds if and only if the distributions $f(TM)$ and $\ker f$ are both harmonic.
%Remark that the class of nearly K\"{a}hler $f$-manifolds contains the class of \textit{Killing $f$-manifolds}, i.e., $(\nabla_X\,f)X=0$,
%which can be defined by the condition that the fundamental form $F(X,Y)=\<X,fY\>$ is a Killing form.
%%
%For a metric $f$-structure with $p=1$, we have two appropriate cases.
%The class of almost contact metric manifolds $(M,g,\phi,\xi,\eta)$ with the condition \eqref{E-cond-PP-stat-2}(a) for $P=\phi$ includes:
%\begin{itemize}
%\item \textit{approximately cosymplectic} manifolds
%(defined by $(\nabla_{X}\,\phi)X=0$ for $X\in\mathfrak{X}_M$;
%\,\footnote{H. Endo, Some remarks of nearly cosymplectic manifolds of constant $\phi$-sectional curvature. {\em Tensor (N.S.)}, 68 (2007),  204--221.};
%, see~\cite{En-2007};
%such non-cosymplectic structure exists, e.g., on $S^7$)
%with parallel Reeb vector field, $\nabla\xi=0$.
%\item \textit{nearly Sasakian manifolds} \cite{b2010},
%i.e., $(\nabla_X\,\phi)X=\<X,X\>\xi -\eta(X)X$, with parallel $\xi$.
%Unfortunately, any nearly Sasakian manifold of dimension $>5$ is~Sasakian.
%%\,\footnote{B. Cappelletti-Montano, A. De Nicola, G. Dileo, I. Yudin, Nearly Sasakian manifolds revisited. {Complex Manifolds}, 6  (2019) 320--334.}.
%\end{itemize}
\end{example}

The next proposition shows that $\bar\nabla^{\ast{P}}$
%, under certain conditions,
is $L^2$-{adjoint to the $P$-derivative} on $k$-forms.

\begin{proposition}
%\label{P-PPcond}
If condition \eqref{E-cond-PP-stat-2} hold for a statistical $P$-connection $\nabla^{P}$,
then
%$\bar\nabla^{*{P}}$ is $L^2$-adjoint to $\nabla^{P}$ on forms: namely,
for any compactly supported
%$(s,t)$-tensor $S_1$ and $(s,t+1)$-tensor $S_2$,
$k$-form $\omega_1$ and $k+1$-form $\omega_2$,
we~have
\begin{equation}\label{E-cond-PP-int}
 (\bar\nabla^{*P} \omega_2,\ \omega_1)_{L^2} = (\omega_2,\ \nabla^{P} \omega_1)_{L^2}.
\end{equation}
\end{proposition}

\begin{proof}
Define a compactly supported 1-form $\omega$ by
\[
 \omega(Y) = \<\iota_{\,Y}\,\omega_2,\,\omega_1\>,\quad Y\in{\mathcal X}_{M}.
\]
It was shown in \cite[Proposition~1]{rp-2} using assumption $\Div P=0$ that
\begin{equation}\label{E-cond-PP-hat}
 -\widehat\nabla^{\ast{P}}\omega = -\<\widehat\nabla^{*P} \omega_2,\ \omega_1\> + \<\omega_2,\ \widehat\nabla^{P} \omega_1\>.
\end{equation}
%Take a local orthonormal frame $(e_{i})$ such that $\nabla_Y\, e_{i}=0$ for all $Y\in T_xM$ at a point $x\in M$.
To~simplify further calculations, assume that $k=1$.
%$s=t=1$.
Then, using \eqref{E-adj-nabla-P} and \eqref{E-cond-PP-hat}, we obtain
\begin{equation}\label{E-divP-omega}
 -\nabla^{\ast{P}}\omega = -\<\bar\nabla^{*P} \omega_2,\ \omega_1\> + \<\omega_2,\ \nabla^{P} \omega_1\>
 +\sum\nolimits_{i\ne j}\<\omega_2(e_i,e_j),\, \omega_1(K_{e_i} e_j)\> ,
\end{equation}
where $(e_{i})$ is a local orthonormal frame on $M$.
By symmetry of $K$ and skew-symmetry of $\omega_2$, the last term in \eqref{E-divP-omega} vanishes.
By \eqref{E-divP-omega}, \eqref{E-divPX} and Theorem~\ref{T-P-Stokes} with $X^\flat=\omega$, we obtain \eqref{E-cond-PP-int}.
\end{proof}

The differential operator $\bar\nabla^{\ast{P}}\nabla^{P}$ will be called the $P$-\textit{Bochner Laplacian} for a statistical $P$-structure.
 The next maximum principle
% for the modified Bochner Laplacian $\bar\nabla^{\ast{P}}\nabla^{P}$
 generalizes ones used in the past.

\begin{proposition}\label{P-max}
Let condition \eqref{E-cond-PP-stat} hold for a statistical $P$-connection $\nabla^{P}$
%for $P\in\mathrm{End}(TM)$ and (1,2)-tensor $K$
on a closed
Riemannian manifold $(M,g)$. Suppose that $\omega$ is a $k$-form
%smooth tensor field
such that
$\<\bar\nabla^{\ast{P}}\nabla^{P} \omega,\,\omega\>\le0$.
Then, $\omega$ is $P$-parallel.
\end{proposition}

\begin{proof}
We apply formula \eqref{E-cond-PP-int},
% and \eqref{E-def-P-LLap},
\[
 0\ge (\bar\nabla^{*P}\nabla^{P}\omega,\,\omega)_{L^2} = (\nabla^{P}\omega,\,\nabla^{P}\omega)_{L^2} \ge 0;
\]
hence, $\nabla^P \omega=0$.
\end{proof}

% and Beltrami
\section{The modified Hodge Laplacian}
% for statistical $P$-structure
\label{sec:laplace}

Using a statistical $P$-connection $\nabla^P$, we define the \textit{exterior ${P}$-derivative} of a differential form $\omega\in\Lambda^{k}(M)$ by
\begin{equation}\label{E-dP}
 d^{P}\omega(X_{0},\ldots,X_{k})
 =\sum\nolimits_{i}(-1)^{i}(\nabla^{P}_{X_{i}}\omega)(X_{0},\ldots,\widehat{X_{i}},\ldots X_{k}).
\end{equation}
%Note that $\nabla^{P}\omega$ given in \eqref{E-nablaP-omega} is not skew-symmetric, but $d^{P}\omega$ is.
For a $k$-form $\omega_{p}$, the $(k+1)$-form $\nabla^{P}\omega$, see \eqref{E-nablaP-S},
\begin{equation*}
%\label{E-nablaP-omega}
 (\nabla^{P} \omega)(Y,X_{1},\ldots,X_{k})
 ={P}Y(\omega(X_{1},\ldots,X_{k})) -\sum\nolimits_{i=1}^{k} \omega(X_{1},\ldots,\nabla_{Y}^{P}X_{i},\ldots,X_{k})
\end{equation*}
is not skew-symmetric,
but the form $d^{P}\omega$ is skew-symmetric.
For a function $f$ on $M$, we have $d^{P}f=\nabla^P f$ and $\bar d^{\,P}f=\bar\nabla^P f$.
% and $\bar d^{P}f=\bar\nabla^P f$.
%, see Definition~\ref{D-31}.
%Similarly we define $\bar d^{\,P}$ using $\bar\nabla^P$.
The next proposition and Remark~\ref{R-02} show that $\bar d^{\,P}= d^{P}$
for statistical $P$-structures.

\begin{proposition}
The $d^{P}:\Omega^{k}(M)\rightarrow\Omega^{k+1}(M)$ is a 1-degree derivation, see Section~\ref{sec:algebroid}, that is
\begin{align}\label{E-1deg-der}
\nonumber
 & d^{P}\omega(X_{0},\ldots,X_{k})=\sum\nolimits_{i=0}^{k}(-1)^{i}{P}X_{i}(\omega(X_{0},\ldots,\widehat{X_{i}},\ldots,X_{k}))\\
 & +\sum\nolimits_{0\le i<j\le k}(-1)^{i+j}\omega\big([X, Y]_{P}, X_{0},\ldots,\widehat{X_{i}},\ldots,\widehat{X_{j}},\ldots,X_{k}\big).
\end{align}
\end{proposition}

\begin{proof} Using \eqref{E-dP} and \eqref{E-nablaP-S} with $s=0$, we obtain
\begin{align*}
& d^{P}\omega(X_{0},\ldots,X_{k})
%=\sum\nolimits_{i=0}^{k}(-1)^{i}(\nabla_{X_{i}}^{P}\omega)(X_{0},\ldots,\widehat{X_{i}},\ldots,X_{k})\\ &
 =\sum\nolimits_{i=0}^{k}(-1)^{i}{P}X_{i}(\omega(X_{0},\ldots,\widehat {X_{i}},\ldots,X_{k}))\\
& +\sum\nolimits_{i=0}^{k}(-1)^{i}\Big(\sum\nolimits_{j=0}^{i-1}
 \omega(X_{0},\ldots,\nabla_{X_{i}}^{P}X_{j},\ldots,\widehat{X_{i}},\ldots,X_{k})\\
&  +\sum\nolimits_{j=i+1}^{k}\omega(X_{0},\ldots,\widehat{X_{i}},\ldots,\nabla_{X_{i}}^{P}X_{j},\ldots,X_{k})\Big)\\
&  =\sum\nolimits_{i=0}^{k}(-1)^{i}{P}X_{i}(\omega(X_{0},\ldots,\widehat{X_{i}},\ldots,X_{k}))\\
&  +\sum\nolimits_{0\leq i<j\leq k}(-1)^{i+j}\omega\big(\nabla_{X_{i}}^{P}X_{j}
-\nabla_{X_{j}}^{P}X_{i},X_{0},\ldots,\widehat{X_{i}},\ldots,\widehat{X_{j}},\ldots,X_{k}\big).
\end{align*}
Using \eqref{E-Pbracket}, we complete the proof of \eqref{E-1deg-der}.
\end{proof}

Put $\delta^{P}=\nabla^{\ast{P}}$ for the $P$-codifferential $\delta^{P}:\Lambda^{k}(TM)\to\Lambda^{k-1}(TM)$.
Similarly, we define
\[
 \bar\delta^{\,P}\omega(X_{2},\ldots,X_{k})=-\sum\nolimits_{\,i}(\bar{\nabla}_{e_{i}}^{P}\,\omega)(e_{i},X_{2},\ldots,X_{k}).
\]

\begin{proposition}\label{P-dP-deltaP}
On a closed $(M,g)$ with a statistical $P$-structure, the $P$-codifferential $\bar\delta^{\,P}$ is $L^2$-adjoint to $d^{P}$, i.e.,
for any differential forms $\omega_{1}\in\Lambda^{k}(TM)$ and $\omega_{2}\in\Lambda^{k+1}(TM)$ we have
\begin{equation}\label{E-1deg-2}
 (\bar\delta^{\,P}\omega_{2},\,\omega_{1})_{L^2} =(\omega_{2},\,d^{P}\omega_{1})_{L^2}.
\end{equation}
\end{proposition}

\begin{proof} We derive
% $(\nabla^{P}\omega_{1},\,\omega_{2})_{L^2} = (\omega_{1},\,\bar\delta^{\,P}\omega_{2})_{L^2}$,
%which requires \eqref{E-cond-PP-stat}, and
\begin{eqnarray*}
 &&\hskip-5mm\<d^{P}\omega_{1},\,\omega_{2}\> = \sum\nolimits_{\,u=0}^{k}(-1)^{i}
 \nabla_{\partial_{i_{u}}}^{P}\omega_{1}(\partial_{i_{1}},\ldots,\widehat\partial_{i_{u}},\ldots,\partial_{i_{k}})
 g^{i_{0}j_{0}}\ldots g^{i_{k}j_{k}}\omega_{2}(\partial_{i_{1}},\ldots,\partial_{i_{k}}) \\
 && = (k+1)\big(\nabla_{\partial_{i_{0}}}^{P}\omega_{1}(\partial_{i_{1}},\ldots,\partial_{i_{k}})\big)
 g^{i_{0}j_{0}}\ldots g^{i_{k}j_{k}}\omega_{2}(\partial_{j_{0}},\ldots,\partial_{j_{k}})
 =\<\nabla^{P}\omega_{1},\,\omega_{2}\>,
\end{eqnarray*}
as in the classical case.
It appears a $(k + 1)$ factor, that finally is absorbed in the definition of~$d^P$.
%Thus, $(d^{P}\omega_{1}, \omega_{2})_{L^2} =(\nabla^{P}\omega_{1},\,\omega_{2})_{L^2}$.
Using this and \eqref{E-cond-PP-int}, which requires \eqref{E-cond-PP-stat}, we obtain \eqref{E-1deg-2}.
%condition \eqref{E-cond-PP-stat} and $\bar d^{\,P}= d^{P}$, by \eqref{E-cond-PP-int} we get
%$(\nabla^{P}\omega_{1},\,\omega_{2})_{L^2} =(\omega_{1},\,\bar\nabla^{P*}\omega_{2})_{L^2}
%=(\omega_{2},\,d^{P}\omega_{1})_{L^2}$.
\end{proof}

%We have $\nabla^{P}=\widehat\nabla^{P}+K$ and its conjugate is $\bar{\nabla}^{P}=\widehat\nabla^{P}-K$.
%We consider the divergence type operators
%\begin{align*}
% \delta^{P}\omega\left(  X_{1},\ldots,X_{k-1}\right) &
% =-\sum\nolimits_{\,i}\nabla_{e_{i}}^{P}\omega\left(  e_{i},X_{1},\ldots,X_{k-1}\right)  ,\\
%%
% \bar\delta^{P}\omega\left(  X_{1},\ldots,X_{k-1}\right) &
% =-\sum\nolimits_{\,i}\bar{\nabla}_{e_{i}}^{P}\omega\left(  e_{i},X_{1},\ldots,X_{k-1}\right) .
%\end{align*}

\begin{definition}\rm
Define the \textit{Hodge type Laplacians} $\Delta_H^{P}$ and $\bar\Delta^{P}_H$ for differential forms $\omega$ by
\begin{eqnarray}\label{E-def-P-Lap}
 \Delta_H^{P}\,\omega = d^{P}\bar\delta^{\,P}\omega +\bar\delta^{\,P}d^{P}\omega,\quad
 \bar\Delta^{P}_H\,\omega = d^{P}\delta^{P}\omega + \delta^{P}d^{P}\omega.
\end{eqnarray}
A differential form $\omega$ is said to be ${P}$-\textit{harmonic} if $\Delta_H^{P}\,\omega=0$
and
$\|\omega\|_{L^2}<\infty$ (similarly for $\bar{P}$).
%$\bar{P}$-\textit{harmonic} if $\bar\Delta_H^{P}\,\omega=0$.
\end{definition}

\begin{remark}\rm
The $P$-harmonic forms have similar properties as in the classical case, e.g., (\cite[Lemma~9.1.1]{Peter}).
Let condition \eqref{E-cond-PP-stat} hold on a closed $(M,g)$.
For $\omega\in\Lambda^k(TM)$, using Proposition~\ref{P-dP-deltaP} and \eqref{E-def-P-Lap}, we~have
\[
 (\Delta_H^{P}\,\omega,\,\omega)_{L^2} =(d^{P}\omega,\,d^{P}\omega)_{L^2}
 + (\bar\delta^{\,P}\omega,\,\bar\delta^{\,P}\omega)_{L^2} ,
\]
thus, $\omega$ is ${P}$-harmonic (and similarly for $\bar P$-harmonic) if and only if $d^{P}\omega=0$ and $\bar\delta^{\,P}\omega=0$.
Observe that, if $\Delta_H^{P}\,\omega=0$ and $\omega=d^{P}\,\theta$, then $\bar\delta^{\,P}d^{P}\theta=\bar\delta^{\,P}\omega=0$.
It follows that
\[
 (\omega,\,\omega)_{L^2} = (d^{P}\theta,\,d^{P}\theta)_{L^2} = (\theta,\,\bar\delta^{\,P}d^{P}\theta)_{L^2}
 = (\theta,\,\bar\delta^{\,P}\omega)_{L^2} =0.
\]
Thus, if $\omega\in\Lambda^k(TM)$ is $P$-harmonic and $\omega=d^{P}\,\theta$ for some $\theta\in\Lambda^{k-1}(TM)$, then $\omega=0$.
\end{remark}

We also consider the Hodge type Laplacian related to $\widehat\nabla^P$, defined in \cite{rp-2} by
\begin{equation*}
 \widehat\Delta^{P}_H = \widehat\delta^{\,P}\,\widehat d^{\,P}+\widehat d^{\,P}\,\widehat\delta^{\,P},
\end{equation*}
where
\begin{eqnarray*}
%\label{E-dP}
 && \widehat d^{\,P}\omega(X_{0},\ldots,X_{k})=\sum\nolimits_{i}(-1)^{i}(\nabla_{PX_{i}}\,\omega)(X_{0},\ldots,\widehat{X_{i}},\ldots X_{k}),\\
 && \widehat\delta^{\,P}\omega(X_{2},\ldots,X_{k}) =-\sum\nolimits_{\,i}({\nabla}_{Pe_{i}}\,\omega)(e_{i},X_{2},\ldots,X_{k}).
\end{eqnarray*}
Similarly to \cite[Eqs. (58) and (59)]{opozda1}, we can state the following

\begin{lemma} For a statistical $P$-structure the following equalities are satisfied:
\begin{eqnarray}\label{E-58-59}
\nonumber
 && \widehat d^{\,P}= d^{P} = \bar d^{P},\\
\nonumber
 && \widehat\delta^{\,P}=\delta^{P}-\iota_{\,E} = \bar\delta^{P} +\iota_{\,E},\\
 && \widehat\Delta^{P}_H=\Delta^{P}_H + {\cal L}^{P}_{E} = \bar\Delta^{P}_H - {\cal L}^{P}_{E},
\end{eqnarray}
where $\mathcal{L}^{P}:=d^{P}\circ\iota-\iota\circ d^{P}$ is the modified Lie derivative.
\end{lemma}

\begin{proof}
From \eqref{E-1deg-der} and \eqref{E-bracket-stat} we get equalities \eqref{E-58-59}$_1$ (for $d^P$).
Next, we obtain
\begin{equation*}
 \delta^{P}\omega = -\sum\nolimits_{\,i}\nabla_{e_{i}}^{P}\iota_{e_{i}}\omega =
 %-\sum\nolimits_{\,i}(\widehat\nabla_{e_{i}}^{P}+K_{e_{i}})\iota_{e_{i}}\omega \\ \eq
 -\sum\nolimits_{\,i}\widehat\nabla_{e_{i}}^{P}\,\iota_{e_{i}}\omega -\sum\nolimits_{\,i}K_{e_{i}}\iota_{e_{i}}\omega
 =\widehat\delta^{\,P}\omega +\iota_{\,E}\,\omega .
\end{equation*}
For the second term, we have used \eqref{E-stat-L1-a}.
From this and $\bar{\nabla}=\widehat\nabla-K$ the equalities \eqref{E-58-59}$_2$ follow.
Finally, we calculate
 %that \eqref{E-58-59}$_2$ follows from
 the following:
\begin{equation*}
 \Delta_{H}^{P} = d^{P}\,\bar{\delta}^{P}+\bar{\delta}^{P}\,d^{P}=d^{P}(\widehat\delta^{\,P}-\iota_{\,E}) +(\widehat\delta^{\,P}-\iota_{\,E}) d^{P}
 %\\ \eq d^{P}\,\widehat\delta^{\,P}+\widehat\delta^{\,P}\,d^{P}-d^{P}\iota_{\,E}-\iota_{\,E}\,d^{P}
 =\widehat{\Delta}^{\,P}-\mathcal{L}_{E}^{P}.
\end{equation*}
From this and $\bar{\nabla}=\widehat\nabla-K$ equalities \eqref{E-58-59}$_3$ follow.
%The second equality follows as previously.
\end{proof}

The following proposition extends result
for regular case, $P=\operatorname{id}_{\,TM}$ and $K=0$ in \cite{csc2010}.

\begin{proposition}
%[see \cite{csc2010} for regular case, $P=\operatorname{id}_{\,TM}$ and $K=0$]
\label{L-Div-1}
Let $(M,g)$ be a complete open Riemannian manifold endowed with
a vector field $X$ such that ${\rm div}_{P}\,X\ge0$ $($or ${\rm div}_{P}\,X\le0)$,
where $P\in\mathrm{End}(TM)$ such that conditions \eqref{E-cond-PP-stat} and $\|PX\|_{g}\in\mathrm{L}^{1}(M,g)$ hold.
Then, ${\rm div}_{P}\,X\equiv0$.
\end{proposition}

\begin{proof}
Let $\omega$ be the $(n-1)$-form in $M$ given by $\omega= \iota_{PX}\,d\operatorname{vol}_{g}$,
i.e., the contraction of the volume form
$d\operatorname{vol}_{g}$ in the direction of $PX$.
If $\{e_{1}, \ldots, e_{n}\}$ is an orthonormal frame on an open set $U\subset M$
with coframe ${\omega_{1}, \ldots, \omega_{n}}$, then
\[
 \iota_{PX}\,d\operatorname{vol}_{g} = \sum\nolimits^{n}_{i=1} (-1)^{i-1}
 \<PX, e_{i}\>\,\omega_{1}\wedge\ldots\wedge\widehat\omega_{i} \wedge\ldots\wedge\omega_{n}.
\]
Since the $(n - 1)$-forms $\omega_{1}\wedge\ldots\wedge\widehat\omega_{i}
\wedge\ldots\wedge\omega_{n}$ are orthonormal in $\Omega^{n-1}(M)$, we get
%\[
 $\|\omega\|_{g}^{2} = \sum\nolimits^{n}_{i=1} \<PX, e_{i}\>^{2} =\|PX\|_{g}^{2}$.
%\]
Thus, $\|\omega\|_{g}\in\mathrm{L}^{1}(M,g)$ and
\[
 d\omega=d(\iota_{PX}\,d\operatorname{vol}_{g})=({\rm div}_{P}\,X)\,d\operatorname{vol}_{g},
\]
see~\eqref{E-divf-4}.
There exists a sequence of domains $B_{i}$ on $M$ such that $M=\bigcup_{\,i\ge1} B_{i}$,
$B_{i}\subset B_{i+1}$ and  $\lim_{\,i\to\infty}\int\nolimits_{B_{i}} d\omega=0$, see \cite{yau}.
Then
\[
 \int_{B_{i}} ({\rm div}_{P}\,X)\,d\operatorname{vol}_{g}
 \overset{\eqref{E-divf-4}}=\int_{B_{i}}{\rm div}(PX)\,d\operatorname{vol}_{g} = \int_{B_{i}} d\omega\to0.
\]
But since ${\rm div}_{P}\,X \ge 0$ on $M$, it follows that ${\rm div}_{P}\,X = 0$ on $M$.
\end{proof}

We call $\Delta^{P}_H\,f = \Div_P(\nabla^P f)$ the $P$-\textit{Laplacian for functions}.
% For any $x\in M$,
Using \eqref{nablaP}, we have
\begin{equation}\label{R-4-2}
 \Delta^{P}_H\,f = \widehat{\Delta}^{\,P}_H\,f + (PE)(f).
% -\mathcal{L}_{E}^{P}(f)(x)
% = \sum\nolimits_{\,i} ({P} e_{i})^2(f)(x) -(PE)(f)(x),
%                    = \sum\nolimits_{\,i}(Pe_{i})^{2}(f)(x) -(PE)(f)(x),
\end{equation}
%where $(e_{i})$ is an orthonormal frame such that $\nabla_Y\,e_{i}=0$ for all $Y\in T_xM$.
%A function $f\in C^\infty(M)$ is called $P$-\textit{subharmonic} if $\Delta^{P}\,f\le0$.
%If $\Delta^{P}\,f\ge0$, then $f$ is called $P$-\textit{superharmonic}.

\smallskip

Consider the following system of singular distributions on a smooth manifold $M$:
${\cal D}_1={\cal D}$,
${\cal D}_2={\cal D}_1+ [{\cal D},{\cal D}_1]$, etc.
The distribution ${\cal D}$ is said to be \textit{bracket-generating} of the step $r\in\NN$
if ${\cal D}_r=TM$, e.g., \cite{CC}.
Note that integrable distributions, i.e., $[X,Y]\in\mathfrak{X}_{\cal D}\ (X,Y\in\mathfrak{X}_{\cal D})$, are not bracket-generating.
The condition $\nabla^{P}f=0$ means that $f\in C^2(M)$ is constant along the (integral curves of)
${\cal D}$; moreover, if ${\cal D}$ is bracket-generating then $f={\rm const}$ on $M$.

The next theorem extends the well-known classical result (see also \cite{csc2010} for $P=\operatorname{id}_{\,TM}$ and $K=0$).
% on subharmonic functions.

\begin{theorem}\label{T-Delta-f}
Let conditions \eqref{E-cond-PP-stat}
%and $PE=0$
hold for a statistical $P$-connection $\nabla^{P}$,
and let $f\in C^2(M)$ satisfy either $\Delta^{P}_H\,f\ge0$ or $\Delta^{P}_H\,f\le0$.
Suppose that any of the following conditions hold:

a)~$(M,g)$ is closed;

b)~$(M,g)$ is open complete, $\|{P}\nabla^{P} f\|$ and $\|f\,{P}\nabla^{P} f\|$ belong to $L^1(M,g)$.

\noindent
Then, $\nabla^{P} f=0$; moreover, if $P(TM)$ is bracket-generating, then $f={\rm const}$.
\end{theorem}

\begin{proof} Set $X=\nabla^{P}f$, then $\Delta^{P}_H\,f=\Div_{P}X$.

a)~Using Theorem~\ref{T-P-Stokes}, we get $\Delta^{P}_H\,f\equiv0$. By the equality with $Y=\nabla^{P} f$,
\begin{equation}\label{E-div-formula}
 {\rm div}_{P}(f\cdot Y) = f\cdot{\rm div}_{P}\,Y + \<\nabla^{P} f,\ Y\>
\end{equation}
and again Theorem~\ref{T-P-Stokes} with $X=f\nabla^{P} f$, we get
$(\nabla^{P} f,\, \nabla^{P} f)_{L^2} = 0$, hence $\nabla^{P} f = 0$.
%%
%For open complete $(M,g)$,

b)~By Proposition~\ref{L-Div-1} with $X=\nabla^{P} f$ and condition $\|{P}\nabla^{P} f\|\in L^1(M,g)$,
we get $\Delta^{P}_H\,f\equiv0$. Using \eqref{E-div-formula} with $Y=\nabla^{P} f$,
Proposition~\ref{L-Div-1} with $X=f\nabla^{P} f$ and condition $\|f\,{P}\nabla^{P} f\|\in L^1(M,g)$, we get
$(\nabla^{P} f,\, \nabla^{P} f)_{L^2} = 0$, hence $\nabla^{P} f = 0$.
If the distribution $P(TM)$ is bracket-generating, then using Chow's theorem \cite{Chow} completes the proof for both cases.
\end{proof}

\section{The modified curvature tensor}
\label{sec:R}

\begin{definition}\rm
Define the \textit{second ${P}$-derivative} of an $(s,k)$-tensor $S$ as the $(s,k+2)$-tensor
\begin{equation*}
 (\nabla^{P})^2_{X,Y}\,S
 %(\nabla^{P}\nabla^{P})_{X,Y}\,S
 = \nabla_{X}^{P}(\nabla_{Y}^{P}\,S) -\nabla_{\nabla_{X}^{P}Y}^{P}\,S\,.
\end{equation*}
Define the $P$-\textit{curvature tensor} of $\nabla^{P}$ by
\begin{eqnarray*}
 R^P_{X,Y}\,Z
 \eq
 (\nabla^{P})^2
 %(\nabla^{P}\nabla^{P})
 _{\,X,Y}\,Z
 -(\nabla^{P})^2
 %-(\nabla^{P}\nabla^{P})
 _{\,Y,X}\,Z \\
 \eq \nabla_{X}^P\nabla^P_{Y} Z-\nabla^P_{Y}\nabla^P_{X}Z -\nabla^P_{\,[X,Y]_P}Z,\quad X,Y,Z\in\mathcal{X}_M,
\end{eqnarray*}
see \eqref{E-T-rho}$_2$ with $\rho=P$, and set
\begin{equation}\label{Curv-P}
  R^{P}(X,Y,Z,W) = \<R^{P}_{X,Y}Z,W\>,\quad X,Y,Z,W\in\mathcal{X}_M.
\end{equation}
The $P$-\textit{Ricci curvature tensor} of $\nabla^P$ is defined by the standard way:
\begin{equation}\label{E-Ric0}
 {\rm Ric}^P(X)=\sum\nolimits_{\,i} R^P_{\,X,e_i}\,e_i,\quad
 {\rm Ric}^P(X,Y)=\sum\nolimits_{\,i} R^P(X,e_i,e_i,Y).
\end{equation}
\end{definition}

Since $\nabla^P$ is torsionless, the first Bianchi identity reads as
\begin{equation*}
 \sum\nolimits_{\,\mathrm{cycl.}}R^P_{\,X,Y} Z = \mathcal{J}_{P}(X,Y,Z) ,
\end{equation*}
where $\mathcal{J}_{P}(X,Y,Z)=\sum\nolimits_{\,\mathrm{cycl.}}[X,[Y,Z]_{P}\,]_{P}$
is called the \emph{Jacobiator} of $[\cdot,\cdot]_{P}$, see \eqref{E-J-rho} in Appendix.
Similarly to \eqref{Curv-S-1}, $R^{P}$ acts on $(0,k)$-tensor fields by
\begin{equation}\label{Curv-S}
  (R_{X,Y}^{P}\,S)(X_1,\ldots,X_k) =\mathfrak{D}^{P}(X,Y)(S(X_1,\ldots,X_k)) -\sum\nolimits_{\,i}S(X_1,\ldots R_{X,Y}^{P}X_i,\ldots,X_k).
\end{equation}
To simplify the calculations,
%(with the contorsion tensor $K$),
in the rest of the article we assume that the \textbf{tensor $\nabla^PK$ is symmetric},
i.e., the following Codazzi type condition:
% and since the contorsion tensor $K$ is symmetric,
%the following conditions:
\begin{eqnarray}
%\labe+l{E-cond-PK}
%  P K_XY\eq K_XPY,\quad X,Y\in\mathfrak{X}_M,\\
\label{E-cond-PK2}
  (\nabla_{PX}\,K)_YZ \eq (\nabla_{PY}\,K)_XZ,\quad X,Y,Z\in\mathfrak{X}_M.
\end{eqnarray}
%By \eqref{E-Pbracket-b}, the condition \eqref{E-cond-PK} means that
%$[X,Y]_P=\nabla_{PX}Y-\nabla_{PY}X$; it follows
%the $P$-bracket does not depend on $K$.
Here, $(\nabla_{PX}K)_Y Z=\nabla_{PX}(K_Y Z) -K_{\nabla_{PX}Y}\,Z -K_Y(\nabla_{PX}\,Z)$.
 Note that $[K_X,K_Y]:TM\to TM$ is a skew-symmetric endomorphism for a statistical $P$-structure.

\begin{proposition}\label{P-RP-properties}
For a statistical $P$-structure,
%with condition \eqref{E-cond-PK2},
 we have
\begin{enumerate}
\item $R^{P}_{X,Y} Z = R_{\,PX,PY} Z+[K_X,K_Y](Z)$;\quad $(R^{P}_{X,Y}\,\omega)(Z) = -\omega(R^{P}_{X,Y} Z -[K_X,K_Y](Z))$;
\vskip1mm
hence,  $\<R^{P}_{X,Y} Z,W\> = -\<R^{P}_{X,Y} W,Z\>$,
\vskip1mm
\item $R^{P}_{X,Y}\,f=0$; \quad $R^{P}_{X,Y}\,g=0$;
\vskip1mm
\item for every $(1,k)$-tensor $S$ we have
\begin{eqnarray*}
 && (R^{P}_{X,Y}\,S)(Z_1,\ldots,Z_k) = (R_{\,PX,PY}\,S)(Z_1,\ldots,Z_k) \\
 && +\,[K_X,K_Y](S(Z_1,\ldots,Z_k)) -\sum\nolimits_{\,i} S(Z_1,\ldots [K_X,K_Y](Z_i), \ldots, Z_k).
\end{eqnarray*}
\item $R^{P}(X,Y,Z,W) = R(PX,PY,Z,W)+\<[K_X,K_Y](Z),W\>$;
\vskip1mm
\item $R^{P}(X,Y,Z,W) = -R^{P}(Y,X,Z,W) = -R^{P}(X,Y,W,Z)$,
where $X,Y,Z,W\in\mathcal{X}_M$, $\omega\in\Lambda^1(TM)$ and $f\in C^2(M)$.
\end{enumerate}
\end{proposition}

\begin{proof}
%$\phantom{.}$
%By \eqref{E-condPP} and \eqref{E-cond-PK}, we have for further use
%\[
% P\nabla_{PX}Y+ P(K_XY) =P\nabla_{X}^{P}Y=\nabla_{X}^{P}(PY)=\nabla_{PX}(PY) + K_X\,PY.
%\]
1. Since $P[X,Y]_{P}=[PX,PY]$, see definition of ${\mathfrak D}^P$, we have
\begin{eqnarray*}
 R^{P}_{X,Y} Z \eq R_{\,PX,PY} Z +(\nabla_{PX}K)_Y Z - (\nabla_{PY}K)_X Z +[K_X,K_Y](Z),\\
 (R^{P}_{X,Y}\,\omega)(Z) \eq -\omega(R_{\,PX,PY} Z) +\omega((\nabla_{PX}K)_Y Z - (\nabla_{PY}K)_X Z +[K_X,K_Y](Z)).
\end{eqnarray*}
From this and \eqref{E-cond-PK2} the first claim follows.
Since $[K_X,K_Y]:TM\to TM$ is skew-symmetric, then $R^{P}_{X,Y}$ is also skew-symmetric.

\noindent
2. We calculate
\begin{equation*}
 R^{P}_{X,Y}\,f
% \nabla_{X}^{P}\nabla_{Y}^{P}f-\nabla_{Y}^{P}\nabla_{X}^{P}f-\nabla_{[X,Y]_{P}}^{P}f
 = PX(PY(f)) -PY(PX(f)) -(P[X,Y]_{P}) f  = \mathfrak{D}^{P}(X,Y)f=0.
\end{equation*}
Next, using 1.
%2\eqref{E-cond-PK2},
we obtain
%using $\mathfrak{D}^P=0$ and \eqref{Curv-P},
\begin{eqnarray*}
 && \<R^{P}_{X,Y} Z,W\>
 %= \<\nabla_{X}^{P}\nabla_{Y}^{P}Z-\nabla_{Y}^{P}\nabla_{X}^{P}Z-\nabla_{[X,Y]_{P}}^{P}Z,\, W\>\\
 %&& = \<\nabla_{PX}\nabla_{PY}Z-\nabla_{PY}\nabla_{PX}Z -\nabla_{[PX,PY]}Z
 %-\nabla_{P[X,Y]_{P}}Z
 %+(\nabla_{PX}K)_Y Z - (\nabla_{PY}K)_X Z,\,W\>\\
 %&&
 %= \<\nabla_{PX}\nabla_{PY}Z-\nabla_{PY}\nabla_{PX}Z-\nabla_{[PX,PY]}Z,\,W\>
 = \<R_{\,PX,PY} Z,W\> +\<[K_X,K_Y](Z),\,W\>.
 %+\<(\nabla_{PX}K)_Y Z - (\nabla_{PY}K)_X Z,\,W\>.
\end{eqnarray*}
Similarly, $\<R^{P}_{X,Y} W,\,Z\> = \<R_{\,PX,PY} W,\,Z\>+\<[K_X,K_Y](W),\,Z\>$ .
%+\<(\nabla_{PX}K)_Y W - (\nabla_{PY}K)_X W,\,Z\>$.
By this and \eqref{Curv-S}, we get
\begin{eqnarray*}
 && (R^{P}_{X,Y}\,g)(Z,W) =  -\<R^{P}_{X,Y} Z,W\> -\<Z,R^{P}_{X,Y} W\> \\
 && = (R_{\,PX,PY}\,g)(Z,W) -\<[K_X,K_Y](Z),W\> -\<[K_X,K_Y](W),Z\> \\
 && = (R_{\,PX,PY}\,g)(Z,W) .
\end{eqnarray*}
Using $R_{\,PX,PY}\,g=0$ and the property \eqref{E-stat-K}, we obtain $R^{P}_{X,Y}\,g=0$.

\noindent
3. From the above and \eqref{Curv-S} the claim follows.
%, since the actions of $R^{P}_{X,Y}$ and $R_{\,PX,PY}$ on generators
%(functions, vector and covector fields) were derived before.

\noindent
%4. We have
%\begin{eqnarray*}
% PR^{P}_{X,Y} Z \eq P(\nabla_{X}^{P}\nabla_{Y}^{P} Z-\nabla_{Y}^{P}\nabla_{X}^{P}Z-\nabla_{[X,Y]_{P}}^{P}Z)
% =\nabla_{X}^{P}P\nabla_{Y}^{P}Z-\nabla_{Y}^{P}P\nabla_{X}^{P}Z-\nabla_{\lbrack X,Y]_{P}}^{P}PZ\\
% \eq \nabla_{X}^{P}\nabla_{Y}^{P}PZ-\nabla_{Y}^{P}\nabla_{X}^{P}PZ-\nabla_{[X,Y]_{P}}^{P}PZ = R^{P}_{X,Y} PZ.
%\end{eqnarray*}
4. The equality follows from \eqref{Curv-P} and 1.
% We have
%, using 4,
%\begin{eqnarray*}
% && R^{P^{\ast}P}(X,Y,Z,W) = \<R^{P}_{X,Y} Z,(P^{\ast}P)W\> = \<PR^{P}_{X,Y} Z,\,PW\> \\
% && =
% \<R^{P}_{X,Y} Z,\,W\> = R^{P}(X,Y,Z,W) = R(PX,PY,Z,W).
%\end{eqnarray*}

\noindent
5.
%The first equality follows 4.
Since $R^{P}_{X,Y} Z=-R^{P}_{Y,X} Z$, see 1., the first equality follows. For the second one, we use~2:
\[
 0=(R^{P}_{X,Y}\,g)(Z,Z) =-2\<R^{P}_{X,Y} Z,\,Z\>;
\]
thus, the claim follows from the equality $\<R^{P}_{X,Y}(Z+W),\,Z+W\>=0$.
%\noindent
%6. This follows from 5.
%By 2, we get $\<R^{P}_{X,Y} Z, Z\> +\,\<(\nabla_{PX}K)_Y Z-(\nabla_{PY}K)_X Z,Z\>=0$.
%Hence,
%\begin{eqnarray*}
% \<R^{P}_{X,Y} Z,W\> \eq -\<R^{P}_{X,Y} W,Z\>\\
% \plus \<(\nabla_{PX}K)_Y Z-(\nabla_{Py}K)_X Z,W\>
% +\<(\nabla_{PX}K)_Y W-(\nabla_{Py}K)_X W,Z\>.
%\end{eqnarray*}
%From this and \eqref{E-cond-PK2} the claim follows.
\end{proof}

Similarly, we define the $P$-curvature tensor of the conjugate $P$-connection $\bar\nabla^P$,
\begin{equation*}
 \bar R^P_{X,Y}\,Z
% = (\nabla^{P})^2_{\,X,Y}\,Z - (\nabla^{P})^2_{\,Y,X}\,Z
 =\bar\nabla_{X}^P\bar\nabla^P_{Y} Z-\bar\nabla^P_{Y}\bar\nabla^P_{X}Z -\bar\nabla^P_{\,[X,Y]_P}Z,\quad X,Y,Z\in\mathcal{X}_M.
\end{equation*}
The following curvature type tensor (depending on $P$ only) has been introduced in \cite{rp-2}:
\begin{equation*}
 \widehat R^{\,P}_{X,Y}\,Z
% = (\nabla^{P})^2_{\,X,Y}\,Z - (\nabla^{P})^2_{\,Y,X}\,Z
 =\nabla_{PX}\nabla_{PY} Z-\nabla_{PY}\nabla_{PX}Z -\nabla_{\,P[X,Y]_P}Z,\quad X,Y,Z\in\mathcal{X}_M,
\end{equation*}
Since we assume $\mathfrak{D}^P=0$ then $\widehat R^{\,P}_{X,Y}=R_{\,PX,PY}$ holds.
By the above,
\[
 R^P_{X,Y} = \widehat R^{\,P}_{X,Y} + [K_X,K_Y],\quad
 \bar R^P_{X,Y} = \widehat R^{\,P}_{X,Y} -[K_X,K_Y].
\]
Thus,
\begin{equation*}
  R^P_{X,Y} + \bar R^P_{X,Y} = 2\, \widehat R^{\,P}_{\,X,Y},\qquad
  \<R^P_{X,Y}\,Z,W\> = -\<\bar R^P_{X,Y}\,W,Z\>.
\end{equation*}
Furthermore,
$\overline{\rm Ric}\/^{\,P}(X,Y)={\rm Ric}^{P}(X,Y)$ when \eqref{E-cond-PK2} hold.

 Denote by $\widehat{\rm Ric}\/^{\,P}$ the Ricci tensor of $\widehat\nabla^{P}$, i.e.,
\[
 \widehat{\rm Ric}\/^{\,P}(X,Y) =\sum\nolimits_{i}R(PX, Pe_{i}, e_{i}, Y).
\]

\begin{proposition} For a statistical $P$-structure, we have
\begin{equation}\label{E-Ric-K}
 {\rm Ric}^{P}(X,Y) = \widehat{\rm Ric}\/^{\,P}(X,Y) + \< K_{X}Y,\, E\>-\<K_{X},\,K_{Y}\>.
\end{equation}
Thus, ${\rm Ric}^{P}$ is symmetric if and only if $\,\widehat{\rm Ric}\/^{\,P}$ is symmetric.
\end{proposition}

\begin{proof} Using symmetry of $K$, we have
\begin{eqnarray*}
 && {\rm Ric}^{P}(X,Y) = \sum\nolimits_{i}R^{P}(X, e_{i}, e_{i}, Y)
 =\sum\nolimits_{i}\big(R(PX, Pe_{i}, e_{i}, Y) + \<[K_{X}, K_{e_{i}}](e_{i}), Y\>\big) \\
 && =\widehat{\rm Ric}\/^{\,P}(X,Y) + \sum\nolimits_{i}\<[K_{X}, K_{e_{i}}](e_{i}), Y\>
 = \widehat{\rm Ric}\/^{\,P}(X,Y) + \< K_{X}Y,\, E\>-\<K_{X},\,K_{Y}\>.
\end{eqnarray*}
%We claim that the map $\mathfrak{K}: (X,Y)\rightarrow \sum\nolimits_{i}\<[K_{X}, K_{e_{i}}](e_{i}), Y\>$ is symmetric.
%Indeed,
%\begin{eqnarray*}
% \mathfrak{K}(X,Y) \eq \sum\nolimits_{a}\<[K_{X},K_{e_{a}}](e_{a}), Y\>
% =\sum\nolimits_{a}\< K_{X}K_{e_{a}}(e_{a}) -K_{e_{a}} K_{X}(e_{a}), Y\>\\
% \eq\sum\nolimits_{i,j,a,b}X^{i} Y^{j}( K_{abb}K_{ija}-K_{iab}K_{jab})  .
%\end{eqnarray*}
From the above the claim follows.
%symmetry of $\mathfrak{K}$
% and the total symmetry of $K$.
%We have also $\mathfrak{K}(X,Y)=\< K_{X}Y, E\>-\operatorname*{trace}(K_{X}K_{Y})$, hence, \eqref{E-Ric-K}.
%Using Lemma~\ref{L-04-K} in what follows
%This completes the proof.
\end{proof}

%\begin{lemma}\label{L-04-K}
%\end{lemma}
%\begin{proof} \end{proof}

%\smallskip

The endomorphism $P$ of $TM$ induces endomorphisms ${\cal P}$ and its adjoint~${\cal P}^{\ast}$ of $\Lambda^{2}(TM)$:
\[
 {\cal P}(X\wedge Y) =PX\wedge PY,\quad
 {\cal P}^{\ast}(X\wedge Y)=P^{\ast}X\wedge P^{\ast}Y,
\]
see \cite{rp-2}.
%Indeed,
%\begin{eqnarray*}
% \<{\cal P}(X\wedge Y),\ Z\wedge W\> \eq \left\vert
%\begin{array}[c]{cc}
% \<PX,Z\> & \<PX,W\>\\
% \<PY,Z\> & \<PY,W\>
%\end{array}
%\right\vert \\
% \eq \left\vert
%\begin{array}[c]{cc}
% \<X,P^{\ast}Z\> & \<X,P^{\ast}W\>\\
% \<Y,P^{\ast}Z\> & \<Y,P^{\ast}W\>
%\end{array}
%\right\vert = \<X\wedge Y,\ {\cal P}^{\ast}(Z\wedge W)\>.
%\end{eqnarray*}
The curvature tensor $R_{X,Y}$ can be seen as a self-adjoint linear operator ${\cal R}$
on the space $\Lambda^{2}(TM)$ of bivectors, called the \textit{curvature operator} (note the reversal of $Z$ and $W$):
\begin{equation*}
 \<{\cal R}(X\wedge Y),Z\wedge W\> = R(X,Y,W,Z),\quad
 {\cal R}^\ast = {\cal R}.
\end{equation*}
 Similarly, we consider $R^{P}_{X,Y}=R_{\,PX,PY}+[K_X,K_Y]$ as a linear operator
% ${\cal R}^{P}={\cal R}\circ{\cal P}+[K,K]$,
or as a corresponding bilinear form on $\Lambda^{2}(TM)$.
For this,  using skew-symmetry of $[K_X,K_Y]$ for a statistical $P$-connection, define a linear operator ${\cal K}$ on $\Lambda^{2}(TM)$ by
\[
 \<{\cal K}(X\wedge Y),Z\wedge W\> = \<[K_X,K_Y](Z),W\>,
\]
and observe ${\cal K}^\ast={\cal K}$ (symmetry).
Put ${\cal R}^{P} = {\cal R}\circ{\cal P} +{\cal K}$ and
$\bar{\cal R}^{P} = {\cal R}\circ{\cal P} -{\cal K}$, i.e.,
\begin{eqnarray*}
 && {\cal R}^{P}(X\wedge Y)={\cal R}\circ {\cal P}(X\wedge Y)+{\cal K}(X\wedge Y)={\cal R}(PX\wedge PY)+{\cal K}(X\wedge Y),\\
 && {\cal R}^{P}(X\wedge Y,Z\wedge W) = \<{\cal R}^{\,P}(X\wedge Y),Z\wedge W\>,\\
 &&\bar{\cal R}^{P}(X\wedge Y)={\cal R}\circ{\cal P}(X\wedge Y)-{\cal K}(X\wedge Y)
 ={\cal R}(PX\wedge PY)-{\cal K}(X\wedge Y) ,\\
 && \bar{\cal R}^{P}(X\wedge Y,Z\wedge W) = \<\bar{\cal R}^{\,P}(X\wedge Y),Z\wedge W\>.
\end{eqnarray*}
%Using known properties of ${\cal R}$ and property 4. of $R^P$, we have
%\begin{equation}\label{E-RRP}
% \<\,{\cal R}^{P}(X\wedge Y),Z\wedge W\> = \<\,{\cal R}(PX\wedge PY)+{\cal K}(X\wedge Y),\,Z\wedge W\> = R^P(X,Y,W,Z).
%\end{equation}
% Similarly we define
%$\bar{\cal R}^{P} = {\cal R}\circ{\cal P} -{\cal K}$, i.e.,
%\begin{eqnarray*}
% && \bar{\cal R}^{P}(X\wedge Y) ={\cal R}\circ {\cal P}(X\wedge Y)-{\cal K}(X\wedge Y)
% = {\cal R}(PX\wedge PY) -{\cal K}(X\wedge Y) ,\\
% && \bar{\cal R}^{P}(X\wedge Y,Z\wedge W) = \<\bar{\cal R}^{\,P}(X\wedge Y),Z\wedge W\>.
%\end{eqnarray*}
Following \cite{rp-2}, we also define $ {\cal R}^{P} = {\cal R}\circ{\cal P}$, i.e.,
\begin{eqnarray*}
 && \widehat{\cal R}^{\,P}(X\wedge Y)={\cal R}\circ {\cal P}(X\wedge Y)={\cal R}(PX\wedge PY),\\
 && \widehat{\cal R}^{\,P}(X\wedge Y,Z\wedge W) = \<\widehat{\cal R}^{\,P}(X\wedge Y),Z\wedge W\>.
\end{eqnarray*}
Using known properties of ${\cal R}$ and property 4. of $R^P$, we have
\begin{eqnarray*}
%\label{E-RRP}
%\nonumber
 && \<\,{\cal R}^{P}(X\wedge Y),Z\wedge W\> = \<\,{\cal R}(PX\wedge PY)+{\cal K}(X\wedge Y),\,Z\wedge W\> = R^P(X,Y,W,Z),\\
%\nonumber
 && \<\,\bar{\cal R}^{P}(X\wedge Y),Z\wedge W\> = \<\,{\cal R}(PX\wedge PY)-{\cal K}(X\wedge Y),\,Z\wedge W\> = \bar R^P(X,Y,W,Z),\\
 && \<\,\widehat{\cal R}^{\,P}(X\wedge Y),Z\wedge W\> = \<\,\widehat{\cal R}(PX\wedge PY),\,Z\wedge W\> = \widehat R^{\,P}(X,Y,W,Z) .
\end{eqnarray*}

%\begin{remark}\rm
%The ${\cal R}^{P}$ on $\Lambda^{2}(TM)$ generally is not self-adjoint:
%\[
% ({\cal R}^{P})^{\ast}=({\cal R}\circ{\cal P})^{\ast}={\cal P}^{\ast}\circ{\cal R}\ne {\cal R}\circ{\cal P}={\cal R}^{P}.
%\]
%\end{remark}

\section{The Weitzenb\"{o}ck type curvature operator}
\label{sec:Ric}

Here, we generalize the Weitzenb\"{o}ck curvature operator \eqref{E-Ric} for the case of distributions.

\begin{definition}\rm
 Define the ${P}$-\emph{Weitzenb\"{o}ck curvature operator}
%(depends linearly on the Riemannian curvature tensor)
 on $(0,k)$-tensors $S$ over $(M,g)$ by
\begin{equation}\label{E-Ric-P}
 \Ric^{P}(S)(X_{1},\ldots,X_{k})
 =\sum\nolimits_{\,a=1}^{k}\sum\nolimits_{\,i}(R^{P}_{\,e_{i},X_{a}}\,S)(\underbrace{X_{1},\ldots,e_{i}}_{a},\ldots,X_{k}).
\end{equation}
The~operators $\overline{\Ric}^{\,P}$ and $\widehat{\Ric}^{\,P}$ are defined similarly using $P$-connections $\bar\nabla^P$ and $\widehat\nabla^P$.
\end{definition}

For a differential form $\omega$, the $\Ric^{P}(\omega)$ is skew-symmetric.
Note that $\Ric^{P}$ reduces to ${\rm Ric}^{P}$
% the $P$-Ricci tensor
 when evaluated on (0,1)-tensors, i.e., $k=1$.
%from \eqref{E-Ric-P}, using \eqref{Curv-S}, we obtain
For $k\ge2$, using \eqref{Curv-S}, from \eqref{E-Ric-P} we~get
\begin{eqnarray}\label{E-Ric-Pb}
 \nonumber
 &&\hskip-4mm \Ric^{P}(S)(X_{1},\ldots,X_{k})
 =-2\sum\nolimits_{\,i,j,a;b<a}R^P(e_i,X_a,e_j,X_b)\cdot S(\underbrace{X_1,\ldots,e_j}_{b},\underbrace{\ldots, e_i}_{a-b},\ldots,X_k)\\
 && +\,\sum\nolimits_{\,i,a}{\rm Ric}^P(e_i,X_a)\cdot S(\underbrace{X_1,\ldots, e_i}_{a},\ldots, X_k),
\end{eqnarray}
or, in coordinates,
 $\Ric^P(S)_{i_1,\ldots, i_k}
 =-2\sum\nolimits_{\,a<b} R^P_{\,j\, i_a p\, i_b} S^{\quad j\ \ \ p}_{i_1\ldots\ \ldots\ \ldots\, i_k}
 +\sum\nolimits_{\,a} {\rm Ric}^P_{\,i_a j}\, S^{\quad\, j}_{i_1\ldots\ \ldots\, i_{k}}$.

%For vector fields and 1-forms, the $\Ric^{P}$ is similar to the $P$-Ricci tensor.
%Define the action of tensor $[K, K]$ on $(0,k)$-tensors $S$ over $(M,g)$ by
%\begin{eqnarray*}
%% [K, K](X,Y,Z,W) = \<[K_{X}, K_{Y}](Z), W\> , \\
% (\mathfrak{K} S)(X_{1},\ldots, X_{k})
% %([K, K] S)(X_{1},\ldots, X_{k})
% \eq -2\sum\nolimits_{\,i,j,a;b<a}\<\,[K_{e_{i}}, K_{X_{a}}](e_{j}),X_{b}\>\,
% S(\underset{b}{\underbrace{X_{1},\ldots,e_{j}}},\underset{a-b}{\underbrace{\ldots,e_{i}}}\ldots,X_{k}) \\
% \eq +\sum\nolimits_{\,i,a} \<[K_{e_{i}}, K_{e_j}](e_j),\, X_a\>
% %\mathcal{K}\left(  e_{i},X_{a}\right)  \cdot
% S(\underbrace{X_{1},\ldots,e_{i}}_{a},\ldots,X_{k}).
%\end{eqnarray*}
%Note that if the tensor $S$ is symmetric, then $[K, K] S =0$.

The following lemma represents $\Re^P$ using $\widehat\Re^{\,P}$ and $K$.

\begin{lemma}
For a statistical $P$-structure, let \eqref{E-cond-PP-stat-2} hold. Then we have
\begin{equation}\label{E-Ric-hat-Ric}
 \Re^P=\widehat\Re^{\,P} - \mathfrak{K},
\end{equation}
where the
%nonnegative definite
operator $\mathfrak{K}$ acts
%on $(0,k)$-tensors $S$
on $k$-forms $\omega$ over $(M,g)$ by
\begin{eqnarray}\label{E-K-frak}
\nonumber
 &&\quad (\mathfrak{K}\,\omega)(X_{1},\ldots, X_{k}) =\sum\nolimits_{\,a=1}^{k}\sum\nolimits_{\,j}
 \< K_{X_a},K_{e_{j}}\>\,\omega(\underbrace{X_{1},\ldots, e_j}_{a},\ldots,X_{k})\\
 && +\,2\sum\nolimits_{\,i,j,b<a}\big(\<K_{X_a}{e_{j}},\,K_{X_b}e_{i}\> -\<K_{e_i}e_{j},\,K_{X_a}X_b\>\big)\,
 \omega(\underset{b}{\underbrace{X_{1},\ldots,e_{j}}},\underset{a-b}{\underbrace{\ldots,e_{i}}}\ldots,X_{k}),
 %\big(\sum\nolimits_{\,a=1}^{k} \|K_{X_a}\|^2\big)\,\omega(X_{1},\ldots, X_{k}),
\end{eqnarray}
when $k\ge2$, and
$(\mathfrak{K}\,\omega)(X) = \sum\nolimits_{\,j}\< K_{X},K_{e_{j}}\>\,\omega(e_j)$ when $k=1$.
\end{lemma}

\begin{proof} Using 1. of Proposition~\ref{P-RP-properties} and \eqref{E-Ric-K}, we have
\begin{eqnarray*}
 R^{P}(e_i,X_a,e_j,X_b) \eq\widehat{R}^{\,P}(e_i,X_a,e_j,X_b)+\<[K_{e_i},K_{X_a}](e_j), X_b\>,\\
 {\rm Ric}^{P}(e_i,X_a) \eq \widehat{\rm Ric}\/^{\,P}(e_i,X_a) + \< K_{e_i}X_a,\, E\>-\<K_{e_i},\,K_{X_a}\>.
\end{eqnarray*}
Substituting the above equalities in \eqref{E-Ric-P} (and using linearity in the curvature) yields \eqref{E-Ric-hat-Ric} with
\begin{eqnarray*}
 && (\mathfrak{K}\,\omega)(X_{1},\ldots, X_{k})
 %= 2\sum\nolimits_{\,i,j,a;b<a}\<\,[K_{e_{i}}, K_{X_{a}}](e_{j}),X_{b}\>\,
 %\omega(\underset{b}{\underbrace{X_{1},\ldots,e_{j}}},\underset{a-b}{\underbrace{\ldots,e_{i}}}\ldots,X_{k}) \\
 %&& -\sum\nolimits_{\,i,a} \big(\< K_{e_i}X_a,\, E\>-\<K_{e_i},\,K_{X_a}\>\big)\,\omega(\underbrace{X_{1},\ldots,e_{i}}_{a},\ldots,X_{k}) \\
%%%%%%%%%
 = \sum\nolimits_{\,i,a}\big(\<K_{X_a},K_{e_{j}}\>-\<K_{X_a}e_{j},E\>\big)\,\omega(\underbrace{X_{1},\ldots,e_{i}}_{a},\ldots,X_{k}) \\
 && +\,2\sum\nolimits_{\,i,j;\,b<a}\big(\<K_{X_a}{e_{j}},\,K_{X_b}e_{i}\>-\<K_{e_i}e_{j},\,K_{X_a}X_b\>\big)\,
 \omega(\underset{b}{\underbrace{X_{1},\ldots,e_{j}}},\underset{a-b}{\underbrace{\ldots,e_{i}}}\ldots,X_{k}),
\end{eqnarray*}
%Next, using \eqref{Curv-S}, the first equality in \eqref{E-Ric-P} and symmetry of $K$, we have
%\begin{eqnarray*}
% (\mathfrak{K}\,\omega)(X_{1},\ldots, X_{k})
% \eq \sum\nolimits_{\,a=1}^{k}\sum\nolimits_{\,i}\omega(\underbrace{X_{1},\ldots,  [K_{e_{i}},K_{X_{a}}](e_i)}_{a},\ldots,X_{k})\\
% \eq \sum\nolimits_{\,a=1}^{k}\sum\nolimits_{\,i}\<[K_{e_{i}},K_{X_{a}}](e_i), e_j\>\,\omega(\underbrace{X_{1},\ldots, e_j}_{a},\ldots,X_{k}) \\
%  \eq \sum\nolimits_{\,a=1}^{k}\sum\nolimits_{\,j}\big(\<K_{X_a},K_{e_{j}}\>-\<K_{X_a}e_{j},E\>\big)\,\omega(\underbrace{X_{1},\ldots,e_j}_{a},\ldots,X_{k}),
%\end{eqnarray*}
that is \eqref{E-K-frak} when $E=0$.
\end{proof}

 The following theorem generalizes \eqref{E-Wei0} to the case of distributions.

\begin{theorem}
For a statistical $P$-structure, let \eqref{E-cond-PP-stat-2} hold. Then the following \textit{Weitzenb\"{o}ck type decomposition formula} is~valid:
\begin{equation}\label{E-Wei}
  \Delta_H^{P} = \bar\nabla^{\ast{P}}\nabla^{P} +\Ric^{P}.
\end{equation}
\end{theorem}

\begin{proof} Similarly to the proof of \cite[Theorem~9.4.1]{Peter} for $\omega\in\Lambda^k(TM)$, or \cite[Theorem~2]{rp-2}, we find
\begin{eqnarray*}
 &&d^{P}\bar{\delta}^{P}\omega( X_{1},\ldots,X_{k})
 =-\sum\nolimits_{j}d^{P}\bar{\nabla}_{e_{j}}^{P}\omega(e_{j},X_{1},\ldots,X_{k}) \\
 &&=-\sum\nolimits_{j}d^{P}( \nabla_{e_{j}}^{P}-2K_{j}) \omega(e_{j},X_{1},\ldots,X_{k}) \\
 && =-\sum\nolimits_{j}d^{P}\nabla_{e_{j}}^{P}\omega( e_{j},X_{1},\ldots,X_{k})
 -2( d^{P}\iota_{\,E}\,\omega)( X_{1},\ldots,X_{k}) \\
 &&=\sum\nolimits_{j}\sum\nolimits_{a=0}^{k-1}(-1)^{a}\nabla_{X_{a+1}}^{P} \nabla_{e_{j}}^{P}\omega(e_{j},X_{1},\ldots\widehat{X_{a+1}}\ldots,X_{k})
 -2( d^{P}\iota_{\,E}\,\omega)(X_{1},\ldots,X_{k}) \\
 &&=-\sum\nolimits_{j}\sum\nolimits_{a=0}^{k-1}\nabla_{X_{a+1}}^{P}
 \nabla_{e_{j}}^{P}\omega\big(\underset{a+1}{\underbrace{X_{1},\ldots e_{j}}},\ldots,X_{k}\big) -2(d^{P}\iota_{\,E}\,\omega)(X_{1},\ldots,X_{k}) \\
 &&= -\sum\nolimits_{j,a}((\nabla^{P})^2_{X_{a+1},e_{j}}\,\omega)
 \big( \underset{a+1}{\underbrace{X_{1},\ldots e_{j}}},\ldots,X_{k}\big)
 -2( d^{P}\iota_{\,E}\,\omega)( X_{1},\ldots,X_{k}) ,
\end{eqnarray*}
and
\begin{eqnarray*}
&& \bar{\delta}^{P}d^{P}\omega( X_{1},\ldots,X_{k})
 =\bar{\nabla}^{\ast P}d^{P}\omega( X_{1},\ldots,X_{k})
 =(\nabla^{\ast P}-2\,\iota_{\,E}) d^{P}\omega( X_{1},\ldots,X_{k}) \\
&&= \nabla^{\ast P}( d^{P}\omega)( X_{1},\ldots,X_{k})
-2\,\iota_{\,E}\,d^{P}\omega( X_{1},\ldots,X_{k}) \\
&& = -\sum\nolimits_{j}\nabla_{e_{j}}^{P}(d^{P}\omega)( e_{j},X_{1},\ldots,X_{k}) -2\,\iota_{\,E}( d^{P}\omega) (X_{1},\ldots,X_{k}) \\
&& = -\sum\nolimits_{j}\nabla_{e_{j}}^{P}\nabla_{e_{j}}^{P}\omega( X_{1},\ldots,X_{k}) \\
&& + \sum\nolimits_{j}\sum\nolimits_{a=0}^{k-1}(-1)^{a}\nabla_{e_{j}}^{P}\nabla^{P}_{X_{a+1}}\omega( e_{j},X_{1},\ldots,\widehat
{X_{a+1}},\ldots,X_{k}) -2\,\iota_{\,E}( d^{P}\omega)(X_{1},\ldots,X_{k}) \\
&&= (\nabla^{\ast P}\nabla^{P}\omega)( X_{1},\ldots,X_{k})
+\sum\nolimits_{j,a}( (\nabla^{P})^2_{e_{j},X_{a+1}}\,\omega)( X_{1},\ldots,X_{k})
-2\,\iota_{\,E}( d^{P}\omega)(X_{1},\ldots,X_{k}) .
\end{eqnarray*}
Thus, if \eqref{E-cond-PP-stat} is assumed, then
using $\nabla^{\ast P}\nabla^{P}=(\bar{\nabla}^{\ast P}+2\,\iota_{\,E})\,\nabla^{P}=\bar{\nabla}^{\ast P}\nabla^{P}+2\nabla_{E}^{P}$,
we have
%\[
% \Delta_{H}^{P}\,\omega=\nabla^{\ast P}\nabla^{P}+\Re^{P}-2\mathcal{L}_{E}^{P}.
%\]
%But
%$\nabla^{\ast P}\nabla^{P}=(\bar{\nabla}^{\ast P}+2\,\iota_{\,E})\,\nabla^{P}=\bar{\nabla}^{\ast P}\nabla^{P}+2\nabla_{E}^{P}$,
%hence, in general,
\begin{equation}\label{E-Wei-E}
  \Delta_{H}^{P}\,\omega
% =\nabla^{\ast P}\nabla^{P}\omega+\Re^{P}\omega -2\,\mathcal{L}_{E}^{P}\,\omega
 =\bar{\nabla}^{\ast P}\nabla^{P}\omega+\Re^{P}\omega -2\,\mathcal{L}_{E}\,\omega +2\,\nabla_{E}^{P}\,\omega.
\end{equation}
Using assumption $E=0$, we reduce \eqref{E-Wei-E} to a shorter form \eqref{E-Wei}.
\end{proof}

Next, we extend the well-known {Bochner--Weitzenb\"{o}ck formula} to the case of distributions with a statistical $P$-structure.

\begin{proposition}
For a statistical $P$-structure, let \eqref{E-cond-PP-stat-2} hold.
Then the following modified {Bochner--Weitzenb\"{o}ck formula} for $k$-forms is valid:
\begin{equation}\label{GrindEP-2-6}
 \frac{1}{2}\,\Delta^{P}_H(\,\|\omega\|^{2}) = -\<\Delta_H^{P}\,\omega, \omega\> +\<\Ric^{P}(\omega), \omega\>
 +\|(\nabla^{P} - K)\,\omega\,\|^{2} +\<\mathfrak{K}\,\omega, \omega\>.
\end{equation}
\end{proposition}

\begin{proof} Applying \cite[Proposition~7]{rp-2}, \eqref{R-4-2} and \eqref{E-58-59}$_3$, we find
\begin{eqnarray*}
 && \frac{1}{2}\,{\Delta}_{H}^{P}(\Vert\,\omega\Vert^{2}) -(PE)(\Vert\,\omega\Vert^{2}) =\frac{1}{2}\,\widehat{\Delta}_{H}^{P}(\Vert\,\omega\Vert^{2}) \\
 && = -\<\widehat\Delta_H^{P}\,\omega, \omega\> +\<\widehat\Ric\/^{P}(\omega), \omega\> +\|\widehat\nabla^{P} \omega\,\|^{2} \\
 && = -\<(\Delta_H^{P} + {\cal L}^P_E)\,\omega, \omega\> +\<(\Ric\/^{P}+\mathfrak{K})\,\omega, \omega\> +\|(\nabla^{P}-K)\,\omega\,\|^{2}.
\end{eqnarray*}
Using assumption $E=0$, we reduce the above to a shorter form \eqref{GrindEP-2-6}.
\end{proof}

\begin{remark}\label{R-cn}\rm
a)~For $k=1$, we have $(\mathfrak{K}\,\omega)(X) = \sum_i \<K_X, K_{e_i}\>\,\omega(e_i)$. Thus,
\[
 \<\mathfrak{K}\,\omega,\,\omega\> = \sum\nolimits_{i,j}\<K_{e_i}, K_{e_j}\>\,\omega(e_i)\,\omega(e_j)
 = \|K_{\omega^\sharp}\|^2 \ge 0,
\]
where $\omega^\sharp=\sum_i\omega(e_i)e_i$ for any $\omega\in\Lambda^1(M)$.
%For $k\ge2$, the last term in \eqref{GrindEP-2-6} can be estimated as $|\<\mathfrak{K}(\omega),\omega\>|\le c(n)\,\|K\|^2\,\|\omega\|^2$,
%where $c(n)>0$ is a constant.

b)~If $\omega$ is $P$-harmonic $k$-form on a closed manifold $M$
%, i.e., $\Delta_H^{P}\,\omega=0$,
and
$\<(\Ric^{P}+\mathfrak{K})(\omega),\omega\>\ge 0$,
then $\Delta^{P}_H(\,\|\omega\|^{2})=0$, $(\nabla^{P}-K)\,\omega=0$ and $(\Ric^{P}+\mathfrak{K})\,\omega=0$, see \eqref{GrindEP-2-6}.
%From the last equality follows $K\omega=0$, hence, $\nabla^{P}\,\omega=0$.
By~Theorem~\ref{T-Delta-f}, $\nabla^{P} \|\omega\|=0$; moreover, if $P(TM)$ is bracket-generating, then $\|\omega\|={\rm const}$ on $M$.
\end{remark}
%Hence, $\|\omega\|$ is constant on $M$,

\begin{example}\rm
%Let $k=1$.
For vector fields and 1-forms, $\Ric^{P}$ reduces to the kind of usual Ricci curvature, see \eqref{E-Ric0} and \eqref{E-Ric-Pb}.
We have $\Ric^{P}(\omega)(X)=\omega({\rm Ric}^{P}(X))$ for any $\omega\in\Lambda^1(M)$. Thus, \eqref{E-Wei} reads as
\[
  \Delta_H^{P}\,\omega = \bar\nabla^{\ast{P}}\nabla^{P}\omega +{\rm Ric}^{P}(\omega).
\]
% for \eqref{GrindEP-2-6}.
\end{example}

Next, rewrite the $P$-Weitzenb\"{o}ck curvature operator using an orthonormal basis $(\xi_{a})$ of skew-symmetric transformations
$\mathfrak{so}(TM)$ and give some applications of $\Ric^{P}$.

 For every bivector $X\wedge Y\in\Lambda^{2}(TM)$,
we build a map ${\cal R}^{P}(X\wedge Y) : \mathcal{X}_M\to \mathcal{X}_M$, given by
\begin{eqnarray*}
 \<{\cal R}^{P}(X\wedge Y) Z,W\> \eq \<{\cal R}^{P}(X\wedge Y), W\wedge Z\> = R^P(X,Y,Z,W) \\
 \eq R(PX,PY,Z,W)+\<[K_X,K_Y](Z),W\> .
\end{eqnarray*}
Since bivectors are generators of the vector space $\Lambda^{2}(TM)$,
we obtain in this way a map ${\cal R}^{P}(\xi): \mathcal{X}_M\to \mathcal{X}_M$
(similarly to algebraic curvature operator ${\cal R}(\xi)$).

\begin{lemma}\label{lmab+}
The map ${\cal R}^{P}(\xi)$, where $\xi\in\Lambda^{2}(TM)$, is skew-symmetric:
\[
 \<{\cal R}^{P}(\xi)W,\, Z\> = -\<{\cal R}^{P}(\xi) Z,\, W\>.
\]
\end{lemma}

\begin{proof}
It suffices to check the statement for the generators.
We have, using Proposition~\ref{P-RP-properties},
\begin{eqnarray*}
 && \<{\cal R}^{P}(X\wedge Y) Z,W\> = R(PX,PY,Z,W)+\<[K_X,K_Y](Z),W\> \\
 && = -R(PX,PY,W,Z) -\<[K_X,K_Y](W),Z\> = -\<{\cal R}^{P}(X\wedge Y) W, Z\>.
\end{eqnarray*}
Thus, the statement follows.
\end{proof}

The associated $P$-\textit{curvature operator} is given by
\[
 \<{\cal R}^{P}(X\wedge Y),\, Z\wedge W\> = R(PX,PY,W,Z) -\<[K_X,K_Y](Z),W\>.
\]
We are based on the fact that, if $X$ and $Y$ are orthonormal, then $X\wedge Y$ is a unit bivector,
while the corresponding skew-symmetric operator (a counterclockwise rotation of $\pi/2$ in the plane ${\rm span}(X,Y)$)
has Euclidean norm $\sqrt2$. To simplify calculations, we assume that $\mathfrak{so}(TM)$ is endowed with metric induced from $\Lambda^{2}(TM)$, see, e.g., \cite{Peter2}.
If $L\in\mathfrak{so}(TM)$, then
\begin{equation}\label{E-L-S}
 (L\,S)(X_1,\ldots, X_k) = -\sum\nolimits_{\,i} S(X_1,\ldots,L(X_i),\ldots, X_k).
\end{equation}
Let $\{\xi_{a}\}$ be an orthonormal base of skew-symmetric transformations such that $(\xi_{a})_{x}\in\mathfrak{so}(T_{x}M)$
for $x$ in an open set $U\subset M$.
By \eqref{E-L-S}, for any $(0,k)$-tensor~$S$,
\[
 (\xi_\alpha S)(X_1,\ldots, X_k) = -\sum\nolimits_{\,i} S(X_1,\ldots,\xi_\alpha(X_i),\ldots, X_k);
\]
The ${\cal R}^{P}(X\wedge Y)$
%and ${\cal K}(X\wedge Y)$
on $\Lambda^{2}(TM)$ can be decomposed using $\{\xi_a\}$.

\begin{lemma}\label{L-0-3P1+} We have
\begin{eqnarray*}
 {\cal R}^{P}(X\wedge Y) \eq -\sum\nolimits_{\,\alpha }\big(\<{\cal P}^{\ast}\circ{\cal R}(\xi_{\alpha})X,Y\>
 +\<{\cal K}(X\wedge Y),\xi_{\alpha}\>\big)\xi_{\alpha} \\
 \eq -\sum\nolimits_{\,\alpha}\big(\<{\cal R}(\xi_{\alpha})PX,PY\>+\<{\cal K}(X\wedge Y),\xi_{\alpha}\>\big)\xi_{\alpha}.
\end{eqnarray*}
\end{lemma}

\proof Using $({\cal R}^{P})^{\ast}={\cal P}^{\ast}\circ{\cal R}$ and Lemma~\ref{lmab+}, we have:
\begin{eqnarray*}
 {\cal R}^{P}(X\wedge Y) \eq \sum\nolimits_{\,\alpha }\<{\cal R}^{P}(X\wedge Y),\xi_{\alpha}\>\,\xi_{\alpha}\\
 \eq\sum\nolimits_{\,\alpha}\big(\<{\cal P}^{\ast}\circ{\cal R}(\xi_{\alpha}),X\wedge Y\>
 +\<{\cal K}(X\wedge Y),\xi_{\alpha}\>\big)\xi_{\alpha} \\
 \eq -\sum\nolimits_{\,\alpha}\big(\<{\cal R}(\xi_{\alpha})PX,PY\>
 +\<{\cal K}(X\wedge Y),\xi_{\alpha}\>\big)\xi_{\alpha}.\quad\Box
\end{eqnarray*}

Lemma~\ref{L-0-3P1+} allows us to rewrite the
%Weitzenb\"{o}ck type curvature
operator \eqref{E-Ric-P}.

\begin{proposition}
%\label{P-0-3PP+}
If $S$ is a $(0,k)$-tensor on $(M,g)$, then
\begin{eqnarray*}
 \Ric^{P}(S)=-\sum\nolimits_{\,\alpha}{\cal R}^{P}(\xi_{a})(\xi_{a}S),\qquad
 (\Ric^{P}(S))^{\ast }=\Ric^{P^{\ast }}(S).
\end{eqnarray*}
In particular, if $P$ is self-adjoint, then $\Ric^{P}$
 %on $\Lambda^{2}(TM)$
 is self-adjoint too.
\end{proposition}

\begin{proof}
We follow similar arguments as in the proof of \cite[Lemma~9.3.3]{Peter}:
\begin{eqnarray*}
 &&\quad \Ric^{P}(S)(X_{1},\ldots ,X_{k}) =\sum\nolimits_{\,i,j}({\cal R}^{P}(e_{j}\wedge X_{i})S)
   (\underbrace{X_{1},\ldots ,e_{j}}_i,\ldots ,X_{k}) \\
 &&  =-\sum\nolimits_{\,i,j,\alpha}
   \big(\<{\cal P}^{\ast}\circ{\cal R}(\xi_{\alpha})e_{j},X_{i}\>+\<{\cal K}(e_{j}\wedge X_{i}),\xi_{\alpha}\>\big)
   (\xi_{\alpha}S)(X_{1},\ldots ,e_{j},\ldots ,X_{k}) \\
 &&  =-\sum\nolimits_{\,i,j,\alpha}(\xi_{\alpha}S)(X_{1},\ldots,
   \big(\<{\cal P}^{\ast}\circ{\cal R}(\xi_{\alpha})e_{j},X_{i}\> e_{j}+\<{\cal K}(e_{j}\wedge X_{i}),\xi_{\alpha}\>\big),\ldots ,X_{k}) \\
 && =-\sum\nolimits_{\,i,j,\alpha}(\xi_{\alpha}S)(X_{1},\ldots,\<e_{j},{\cal R}^{P}(\xi_{\alpha})X_{i}\> e_{j},\ldots ,X_{k}) \\
 && =-\sum\nolimits_{\,i,\alpha}(\xi_{\alpha}S)(X_{1},\ldots,{\cal R}^{P}(\xi_{\alpha})X_{i},\ldots ,X_{k})
    =-\sum\nolimits_{\,\alpha}({\cal R}^{P}(\xi_{\alpha})(\xi_{\alpha}S))(X_{1},\ldots ,X_{k}).
\end{eqnarray*}%
Thus, the first claim follows.
Since ${\cal R}:\Lambda^{2}(TM)\to \Lambda^{2}(TM)$ is self-adjoint, there is a local orthonormal base $\{\xi_{a}\}$ of $\Lambda^{2}(TM)$
such that ${\cal R}(\xi_{a})=\lambda_{a}\,\xi_{a}$. Using this base, for any $(0,k)$-tensors $S_1$ and $S_2$, we have
\begin{eqnarray}\label{E-RicP-SS}
\nonumber
 \<\Ric^{P}(S_2),\, S_1\> \eq-\sum\nolimits_{\,\alpha}\<{\cal R}^{P}(\xi_{\alpha})(\xi_{\alpha}S_2),S_1\>
   =-\sum\nolimits_{\,\alpha }\<\,\xi_{\alpha}S_2,\,({\cal R}^{P})^\ast(\xi_{\alpha})S_1\> \\
\nonumber
 \eq \sum\nolimits_{\,\alpha}\<\,\xi_{\alpha}S_2,\,({\cal P}^{\ast}\circ{\cal R}+{\cal K})(\xi_{\alpha})(S_1)\> \\
 \eq \sum\nolimits_{\,\alpha}\lambda_{\alpha}\<\,{\cal P}(\xi_{\alpha}S_2),\,\xi_{\alpha}S_1\>
 + \sum\nolimits_{\,\alpha}\<\,{\cal K}(\xi_{\alpha}S_2),\,\xi_{\alpha}S_1\>,
\end{eqnarray}
and, similarly, again using ${\cal K}^{\ast}={\cal K}$,
\begin{eqnarray*}
 \<S_2,\,\Ric^{P^{\ast}}(S_1)\> = \sum\nolimits_{\,\alpha}\lambda_{\alpha}\<\,\xi_{\alpha}S_2,\,{\cal P}^{\ast}(\xi_{\alpha}S_1)\>
 + \sum\nolimits_{\,\alpha}\<\,\xi_{\alpha}S_2,\,{\cal K}(\xi_{\alpha}S_1)\> \\
 =\sum\nolimits_{\,\alpha}\lambda_{\alpha}\<\,{\cal P}(\xi_{\alpha}S_2),\,\xi_{\alpha}S_1\>
 + \sum\nolimits_{\,\alpha}\<\,{\cal K}(\xi_{\alpha}S_2),\,\xi_{\alpha}S_1\>.
\end{eqnarray*}
Thus, the second claim follows.
\end{proof}

Next, we will extend \cite[Corollary 9.3.4]{Peter} for the case of singular distributions.

\begin{proposition}\label{P-RP-ge0}
Let $(g,\nabla^P)$ be a statistical $P$-structure on a manifold $M$.

a)~If $\<{\cal R}^{P}(S),S\>\ge 0$ for any $(0,k)$-tensor $S$, then $\<\Ric^{P}(S),S\>\ge 0$.

b)~Moreover, if $\<{\cal R}^{P}(S),S\>\ge -\varepsilon\,\|S\|^{2}$ for any $(0,k)$-tensor $S$, where $\varepsilon>0$, then
\[
 \<{\Ric}^{\,P}(S),S\>\ge -\varepsilon\,C\,\|S\|^{2},
\]

where a constant $C$ depends only on the type of~$S$.
\end{proposition}

\begin{proof}
Using \eqref{E-RicP-SS} and a local orthonormal base $\{\xi_{\alpha}\}$ of $\Lambda^{2}(TM)$ such that ${\cal R}(\xi_{\alpha})=\lambda_{\alpha}\xi_{\alpha}$, we~get
\begin{eqnarray*}
 \<\Ric^{P}(S),S\>
 \eq \sum\nolimits_{\,\alpha}\lambda_{\alpha}\<\,{\cal P}(\xi_{\alpha}S),\,\xi_{\alpha}S\>
 +\sum\nolimits_{\,\alpha}\<\,{\cal K}(\xi_{\alpha}S),\,\xi_{\alpha}S\>\\
 \eq \sum\nolimits_{\,\alpha}\<\,{\cal P}(\xi_{\alpha}S),\,{\cal R}(\xi_{\alpha}S)\>
 +\sum\nolimits_{\,\alpha}\<\,{\cal K}(\xi_{\alpha}S),\,\xi_{\alpha}S\> \\
 \eq \sum\nolimits_{\,\alpha}\<\,{\cal R}^P(\xi_{\alpha}S),\,\xi_{\alpha}S\>.
\end{eqnarray*}
By conditions, $\<\,{\cal R}^P(\xi_{\alpha}S),\,\xi_{\alpha}S\>\ge 0$ for all $\alpha$,
thus, $\<\Ric^{P}(S),S\>\ge 0$, and the first claim~follows.
There is a constant $C>0$ depending only on the type of the tensor and $\dim M$ such that
%\[
 $C\|S\|^2\ge \sum\nolimits_{\,\alpha}\|\xi_{\alpha}S\|^2$,
%\]
see~\cite[Corollary~9.3.4]{Peter}.
By conditions, $\<\,{\cal R}^P(\xi_{\alpha}S),\,\xi_{\alpha}S\,\>\ge -\varepsilon\,\|\xi_{\alpha}S\|^2$ for all $\alpha$.
The above
%, for $\varepsilon>0$,
 yields
$\<\,{\cal R}^P(\xi_{\alpha}S),\,\xi_{\alpha}S\>\ge -\varepsilon\,C\,\|S\|^{2}$ -- thus, the second~claim.
\end{proof}

The following result extends \cite[Theorem~3.3 for $P=\id_{\,TM}$]{Peter2}
and \cite[Corollary~1]{rp-2}.

\begin{theorem}\label{T-85}
Let \eqref{E-cond-PP-stat-2} be satisfied for a statistical $P$-structure on a closed manifold $M$ and
$\<{\cal R}^{P}(\omega),\omega\>\ge 0$ for any $k$-form $\omega$. Then any $P$-harmonic $k$-form on $M$ is $P$-parallel.
\end{theorem}

\begin{proof} By conditions and Proposition~\ref{P-RP-ge0}(a), $\<\Ric^{P}(\omega),\omega\>\ge 0$.
By \eqref{E-Wei}, since $\Delta^P_H\,\omega=0$, we get $\<\bar\nabla^{\ast{P}}\nabla^{P}\omega,\omega\>\le0$.
By Proposition~\ref{P-max}, we have $\nabla^P\omega=0$.
\end{proof}

The following result extends \cite[Theorem~3 with $\widehat\nabla^{\,P}$]{rp-2} and $k=1$.

\begin{theorem}\label{T-86}
Let \eqref{E-cond-PP-stat-2} be satisfied for a statistical $P$-connection on an open complete $(M,g)$
and $\|K_{X}\|\ge \varepsilon\,\|X\|$ for some $\varepsilon>0$ and all $X\in TM$.
Suppose that
%$\<(\Ric^{P}+\mathfrak{K})(\omega),\omega\>\ge 0$
%\[
 $\<{\cal R}^{P}(\omega),\omega\>\ge -(\varepsilon/C)\,\|\omega\|^2$
%\]
for any $1$-form $\omega$, where
%$c(n)$ is defined in Remark~\ref{R-cn} and
$C$ is defined in Proposition~\ref{P-RP-ge0}(b).
If~$\,\|{P}\,\nabla^{P}(\|\omega\|^2)\|\in L^1(M,g)$
for a $P$-harmonic $1$-form $\omega$, then
$\widehat\nabla^{\,P}\,\omega=0$.
%$\omega$ is $P$-parallel.
%$\|\omega\|$ is constant and
%$\nabla^{P} \|\omega\|=0$, $(\nabla^{P}-K)\,\omega=0$ and $(\Ric^{P}+\mathfrak{K})\,\omega=0$.
\end{theorem}

\begin{proof}
By conditions, Remark~\ref{R-cn} and Proposition~\ref{P-RP-ge0}(b),
\[
 \<\widehat\Ric^{P}(\omega),\omega\> = \<\Ric^{P}(\omega),\omega\>+\<\mathfrak{K}(\omega),\omega\>
 \ge -\varepsilon\,\|\omega\|^2+\|K_{\omega^\sharp}\|^2 \ge 0.
\]
By
%the Bochner--Weitzen\-b\"{o}ck type formula
\eqref{GrindEP-2-6} with $K=0$,
% with $S=\omega$,
since $\widehat\Delta^P_H\,\omega=\Delta^P_H\,\omega=0$, see \eqref{E-58-59}, we get $\Delta^{P}_H(\|\omega\|^2)\ge0$.
By Proposition~\ref{L-Div-1} with $K=0$ and $X=\widehat\nabla^{\,P}(\|\omega\|^2)$, we get $\widehat\Delta^{P}_H(\|\omega\|^2)=0$.
Applying Theorem~\ref{T-Delta-f}(b), we get $\widehat\nabla^{\,P}\,\omega=0$.
% and $(\Ric^{P}+\mathfrak{K})\,\omega=0$.
%By Theorem~\ref{T-Delta-f}, $\nabla^{P} \|\omega\|=0$.
%; moreover, if $P(TM)$ is bracket-generating, then $\|\omega\|={\rm const}$ on $M$.
%By \eqref{GrindEP-2-6} again, $\|\omega\|$ is constant and $(\nabla^{P}-K)\,\omega=0$.
%% and $\mathfrak{K}\,\omega=0$. From this follows $K\omega=0$.
%%$\nabla^P\omega=0$ and $K\,\omega=0$.
\end{proof}

%Using Remark~\ref{R-cn}, we obtain from Theorem~\ref{T-86} the following.

%\begin{corollary}\label{C-86}
%Let \eqref{E-cond-PP-stat-2} be satisfied for a statistical $P$-connection on an open complete $(M,g)$
%and $\<{\rm Ric}^{P}(\omega),\omega\>\ge0$ for any $1$-form $\omega$. If $\|{P}\,\nabla^{P}(\|\omega\|^2)\|\in L^1(M,g)$
%hold for a $P$-harmonic 1-form $\omega$, then $\widehat\nabla^{\,P}\,\omega=0$.
%\end{corollary}

Notice that, if
%the distribution
$P(TM)$ in Theorems~\ref{T-85} and \ref{T-86}
%and Corollary~\ref{C-86}
is bracket-generating, then $\|\omega\|={\rm const}$
on $M$.

\section{Appendix: the almost Lie algebroid structure}
\label{sec:algebroid}

Lie algebroids (and Lie groupoids) constitute an active field of research in differential geometry.
Roughly speaking, an (almost) Lie algebroid is a structure, where one replaces the tangent bundle $TM$ of a manifold $M$ with a new
smooth vector bundle $\pi_{E}\colon E\to M$ of rank $k$ over $M$ (i.e., a smooth fibre bundle with fibre $\RR^k$) with similar properties.
Many geometrical notions, which involve $TM$, were generalized to the context of Lie algebroids.
Lie groupoids
%(not considered in the paper)
are related to Lie algebroids similarly as Lie groups are related to Lie algebras, see~\cite{mak}.
Lie algebroids deal with integrable distributions (foliations).
Almost Lie algebroids are closely related to singular distributions, e.g. \cite{rp-1,rp-2}.
%Let
%$\pi_{E}\colon E\to M$
%$E$ be a smooth vector bundle of rank $k$ over $M$ (i.e., a smooth fibre bundle with fibre $\RR^k$).
%A \textit{Riemannian bundle} $(E,g)$ over $M$ has a symmetric positive definite (0,2)-tensor field $g=\<\cdot,\cdot\>$,
%i.e., $g_x$ is an inner product in each fiber $E_x$ smoothly depending on $x\in M$.
%The main example is $E=TM$ with the \textit{Riemannian structure (metric)} on $M$.
Here, we recall some facts about this structure on $E$, e.g., \cite{PP01,PMAlg,rp-2}.
%Let $\pi_{E}\colon E\to M$ be a smooth vector bundle of rank $k$ over $M$ (i.e., a smooth fibre bundle with fibre $\RR^k$).
%A \textit{Riemannian bundle} $(E,g)$ over $M$ has a symmetric positive definite (0,2)-tensor field $g=\<\cdot,\cdot\>$,
%i.e., $g_x$ is an inner product in each fiber $e_x$ smoothly depending on $x\in M$.
%The main example is $E=TM$ with the \textit{Riemannian structure (metric)} on $M$.
%see more examples in \cite{GH}.

\begin{definition}\label{D-ALA}\rm
An \textit{anchor} on $E$ is a morphism $\rho\colon E\to TM$ of vector bundles.
 A \textit{skew-symmetric bracket} on $E$ is a map
$[\cdot,\cdot]_{\rho}\colon {\mathfrak X}_E\times\mathfrak{X}_E\to{\mathfrak X}_E$ such~that
\begin{equation}\label{E-anchor}
 [Y,X]_{\rho}=-[X,Y]_{\rho},\qquad
 [X,fY]_{\rho}=\rho(X)(f)Y+f[X,Y]_{\rho},\qquad
 \rho([X,Y]_{\rho})=[\rho(X),\rho(Y)]
\end{equation}
for all $X,Y\in{\mathfrak X}_E$ and $f\in C^\infty(M)$.
The anchor and the skew-symmetric bracket give \textit{an almost Lie algebroid structure} on
%a vector bundle
$E$.
The tensor map $\mathcal{J}_{\rho}: {\mathfrak X}_E\times{\mathfrak X}_E\times{\mathfrak X}_E\to{\mathfrak X}_E$, given by
\begin{equation}\label{E-J-rho}
 \mathcal{J}_{\rho}(X,Y,Z)=\sum\nolimits_{\,\mathrm{cycl.}}[X,[Y,Z]_{\rho}\,]_{\rho},
\end{equation}
which measures a bracket's failure to satisfy the Jacobi identity,
is called the \emph{Jacobiator} of the bracket;
 using \eqref{E-anchor}$_3$, we get $\rho\mathcal{J}_{\rho}=0$.
An almost Lie algebroid is a \textit{Lie algebroid} provided that $\mathcal{J}_{\rho}$
%the Jacobiator of the bracket $[\cdot,\cdot]_\rho$
vanishes.
\end{definition}

Note that axiom \eqref{E-anchor}$_3$ is equivalent to vanishing of the following operator:
\begin{equation}\label{E-D-rho}
  \mathfrak{D}^\rho(X,Y) = [\rho X, \rho Y] - \rho([X,Y]_\rho).
\end{equation}

There is a bijective correspondence between almost Lie algebroid structures on $E$ and the exterior differentials of the exterior algebra $\Lambda(E)={\bigoplus}_{\,k\in\NN}\,\Lambda^{k}(E)$ of the dual bundle $E^{\ast}$, see~\cite{P2};
here $\Lambda^{k}(E)$ is the set of $k$-forms over $E$. The exterior differential $d^\rho$, corresponding to the almost Lie algebroid structure $(E,\rho,[\cdot,\cdot]_\rho)$, is given by
\begin{eqnarray*}
 d^\rho\,\omega(X_{0},\ldots,X_{k}) \eq \sum\nolimits_{\,i=0}^k(-1)^{i}(\rho X_{i})
 (\omega(X_{0},\ldots,\widehat{X_{i}},\ldots,X_{k})) \\
 \plus\mathrel{\mathop{\sum }\nolimits_{\,0\le i<j\le k}}(-1)^{i+j}\omega
 ([X_{i},X_{j}]_\rho,X_{0},\ldots,\widehat{X_{i}},\ldots,\widehat{X_{j}},\ldots,X_{k}),
\end{eqnarray*}
where $X_{0},\ldots,X_{k}\in{\mathfrak X}_E$ and $\omega\in\Lambda^{k}(E)$ for $k\ge0$.
For $k=0$, we have
%\[
 $d^\rho f(X)=(\rho X)(f)$, where $X\in{\mathfrak X}_E$ and $f\in C^\infty(M)=\Lambda^{0}(E)$.
%\]
Recall that a skew-symmetric bracket defines uniquely an exterior differential $d^{\rho}$ on $\Lambda(TM)$,
and it gives rise to

\noindent\
-- a \textit{skew-symmetric algebroid} if and only if $(d^{\rho})^{2}f=0$ for $f\in C^\infty(M)$;

\noindent\
-- a \textit{Lie algebroid} if and only if $(d^{\rho})^{2}f=0$ and
$(d^{\rho})^{2}\,\omega=0$ for $f\in C^\infty(M)$ and $\omega\in\Lambda^{1}(TM)$.

\begin{definition}
%\label{D-rho-conn}
\rm
A \textit{$\rho$-connection} on $(E,\rho)$ is a map $\nabla^\rho:\mathfrak{X}_E\times\mathcal{X}_E\to\mathfrak{X}_E$ satisfying Koszul~conditions
\begin{equation}\label{E-rho-conn}
 \nabla^\rho_{X}\,(fY+Z)=\rho(X)(f)Y +f\nabla^\rho_X\,Y +\nabla^\rho_{X}\,Z,\qquad
 \nabla^\rho_{fX+Z}\,Y=f\nabla^\rho_{X}\,Y+\nabla^\rho_{Z}\,Y.
\end{equation}
For a $\rho$-connection $\nabla^\rho$ on $E$, they define \textit{torsion} $T^\rho:\mathfrak{X}_E\times\mathcal{X}_E\to\mathfrak{X}_E$
and \textit{curvature} $R^\rho:\mathfrak{X}_E\times\mathcal{X}_E\times\mathcal{X}_E\to\mathfrak{X}_E$ by ``usual" formulas
\begin{eqnarray}\label{E-T-rho}
% &&
 T^\rho(X,Y)=\nabla^\rho_X\,Y-\nabla^\rho_Y\,X -[X,Y]_\rho,\qquad
%\label{E-RXY-rho} &&
 R^\rho_{X,Y}\,Z=\nabla_{X}^\rho\nabla^\rho_{Y} Z-\nabla^\rho_{Y}\nabla^\rho_{X}Z -\nabla^\rho_{\,[X,Y]_\rho}Z.
\end{eqnarray}
\end{definition}

The following equality holds, see \cite{PMAlg}:
\begin{equation*}
%\label{E-R-rho}
 \sum\nolimits_{\,\mathrm{cycl.}}R^\rho_{\,X,Y} Z = \sum\nolimits_{\,\mathrm{cycl.}}[(\nabla^\rho_X\,T^\rho)(Y,Z)+T^\rho(T^\rho(X,Y),Z)]
 + \mathcal{J}_{P}(X,Y,Z) .
\end{equation*}

\section{Conclusion}
\label{sec:conclusion}

The main contribution of this paper is the further development of Bochner's technique for a regular or singular distribution
parameterized by a smooth endomorphism $P$ of the tangent bundle of a Riemannian manifold with linear connection.
In particular, the main results of this paper, Theorems 1--6 are proved.
%and to use the concept of almost Lie algebroid
We introduce the concept of statistical $P$-structure, i.e., a pair $(g,\nabla^P)$ of a metric $g$ and $P$-connection $\nabla^P$ on $M$ with totally symmetric contorsion tensor $K$, see \eqref{E-stat-K}, and assume~\eqref{E-condPP-stat} for $P$ to use the theory of almost Lie algebroids. To generalize some geometrical analysis tools for distributions,
we assume the additional conditions \eqref{E-cond-PP-stat-2} and \eqref{E-cond-PK2} for tensors $P$ and $K$.
We introduce a Weitzenb\"{o}ck type curvature operator on tensors, derive a Bochner--Weitzenb\"{o}ck type formula and prove vanishing theorems on the null space of the Hodge type Laplacian on a distribution.
%For the future:

We delegate the following for further study:
a)~generalize some constructions in the paper, e.g., statistical $P$-structures, divergence results, to more general skew-symmetric algebroids or Lie algebroids;
b)~use less restrictive conditions on $K$;
c)~find more applications in geometry and~physics.

%\newpage
%\smallskip
%\noindent
%\textbf{Author Contributions}: All authors contributed equally and significantly in writing this article.

%\noindent
%\textbf{Funding}: This research received no external funding.

%\noindent
%\textbf{Conflicts of Interest}: The authors declare no conflict of interest.

%\baselineskip=11.3pt

%\reftitle{\textrm{\large References}}


\begin{thebibliography}{9}                                                                                                %

\bibitem{Amari2016}
 Amari, S.-I. Information geometry and its applications. Applied Math. Sciences, 194. Springer, 2016.

\bibitem{BF}
 Bejancu, A.; Farran, H. {\em Foliations and Geometric Structures};  Springer: {London, UK}, 2006.%Newly added location, please confirm.
%, 300 pp.

\bibitem{BL}
F.~Bullo and A.D.~Lewis, \textit{Geometric control of mechanical systems:
Modeling, analysis, and design for simple mechanical control systems}, Texts in Applied Math., 49, Springer, 2005.

\bibitem{CC}
 Calin, O. and Chang, D.-C. {\em Sub-Riemannian geometry. General theory and examples}.
Encyclopedia of Mathematics and its Applications, 126. Cambridge University Press, 2009.

\bibitem{csc2010}
A. Caminha, P. Souza, F. Camargo, Complete foliations of space forms by hypersurfaces,
Bull. Braz. Math. Soc., New Series, {41}:3 (2010), 339--353

\bibitem{Chow} Chow, W.-L. \"{U}ber Systeme von linearen partiellen Differentialgleichungen erster Ordnung. {\em Math. Ann.}  {1939}, {117}, 98--105.

\bibitem{g1967}
 Gray, A. {Pseudo-Riemannian almost product manifolds and submersions}. {\em J. Math. Mech.} {16}, {1967}, 715--737.

%\bibitem{gr89}
% Grigor'yan A.A. Stochastically complete manifolds and summable harmonic functions.
% Mathematics of the USSR Izvestiya, Math. USSR Izv. 33 (1989)  425--432.

\bibitem{mak} Mackenzie, K.C.H. \textit{General Theory of Lie Groupoids and Lie Algebroids};
London Math. Soc. Lecture Note Series; Cambridge University Press: {Cambridge, UK}, 2005; Vol. 213.

\bibitem{Mikes1}
%Mikes, J. and Stepanova, E. A five-dimensional Riemannian manifold with an irreducible SO(3)-structure as a model of abstract statistical %manifold. Ann. Global Anal. Geom. 45\,:\,2 (2014), 111--128.
Mikes, J. and  Stepanov, S.E. A Brief Review of Publications on the Differential Geometry of Statistical Manifolds.
COJ Technical \& Scientific Research. 2(4). COJTS.000543.2020.




\bibitem{Mol}
P. Molino, \textit{Riemannian foliations}, Progress in Math., vol. 73, Birkh\"{a}user, 1988

\bibitem{opozda1} Opozda, B. Bochner's technique for statistical structures. Ann. Global Anal. Geom. 48, 357--395 (2015).

\bibitem{pss-2020}
Pan'zhenskii, V.I., Stepanov, S.E. and Sorokina, M.V. Metric-affine spaces, J. of Math. Sciences, 245\,:\,5 (2020), 644--658

\bibitem{Peter} Petersen, P. \textit{Riemannian geometry}, Springer Int. Publ. AG, Switzerland, 2016.

\bibitem{Peter2} Petersen, P. Demystifying the Weitzenb\"{o}ck Curvature Operator.  Preprint is available online: \texttt{http://www.math.ucla.edu/~petersen/} (Accessed 31 Aug. 2015).

\bibitem{PP0}
 Popescu, P.; Popescu, M. {On singular non-holonomic geometry}. {\em Balkan J. of Geom.  Its Appl.}
 {2013}, {\em 18},  58--68.

\bibitem{PMAlg} Popescu, M. and Popescu, P. {Almost Lie algebroids and characteristic classes}.
{Symmetry Integr. Geom. Methods Appl.} {2019}, {15}, 021.

\bibitem{PP01} Popescu, M. and Popescu, P. {Geometrical objects on anchored vector bundles}.
In \emph{Lie Algebroids and Related Topics in Diff. Geometry}; Kubarski, J., Urbanski, P., Wolak, R., eds.;
Banach Center Publ.: {Bedlewo, Poland}, {2001}, volume: {54}, pp. 217--233.

\bibitem{P2} Popescu, P. {Almost Lie structures, derivations and $R$-curvature on relative tangent spaces}.
{Rev. Roum. Math. Pures Appl.} {1992}, {37}, 779--789.

\bibitem{rp-1} Popescu, P. and Rovenski, V. {An integral formula for singular distributions},
Results in Mathematics, 75, Article number: 18 (2019), DOI: 10.1007/s00025-019-1145-1.

\bibitem{rp-2} Popescu, P., Rovenski V. and Stepanov, S. {The Weitzenb\"{o}ck type curvature operator for singular distributions},
Mathematics, MDPI, 2020, 18 pp.

\bibitem{sss-1}
Stepanov, S.E., Stepanova, E.S. and Shandra I.G. Conjugate connections on statistical manifolds.
Russian Math. (Iz. VUZ), 51\,:\,10 (2007), 89--96.

\bibitem{tak}
Takano, K. Statistical manifolds with almost contact structures and its statistical submersions.
J. Geom., 85 (2006) 171--187

\bibitem{vil}
 V\^{i}lcu G.-E. Almost product structures on statistical manifolds and para-K\"{a}hler-like statistical submersions.
 arXiv:1904.09411.

\bibitem{yau}
 Yau S.T., {Some function-theoretic properties of complete Riemannian mani\-folds and their applications to geometry}.
Indiana Univ. Math. J. 25 (1976), 659--670.

\end{thebibliography}
\end{document}